\documentclass[12pt,a4,english,leqno]{article}
\usepackage[latin9]{inputenc}
\usepackage{geometry}
\usepackage{pdflscape}
\usepackage{verbatim}
\usepackage{amsmath}
\usepackage{amsfonts}
\usepackage{amsthm}
\usepackage{mathtools}
\usepackage{graphicx}
\usepackage{longtable}
\usepackage{setspace}
\usepackage{xcolor}
\usepackage{comment}
\usepackage[round,sort,authoryear]{natbib}
\usepackage[colorlinks=true,linkcolor=blue, citecolor=blue, urlcolor = blue]{hyperref}

\usepackage{titlesec}
\usepackage{babel}
\usepackage[toc,page]{appendix}
\usepackage{booktabs,caption}

\usepackage[labelfont=bf,textfont=md]{caption}
\usepackage[labelsep=space]{caption}

\usepackage{lmodern}
\usepackage[T1]{fontenc}

\DeclarePairedDelimiter{\floor}{\lfloor}{\rfloor}

\newif\ifshow 
\showfalse 

\ifshow
  \includecomment{wrap}
\else
  \excludecomment{wrap} 
\fi

\usepackage{bibentry}
\nobibliography*

\numberwithin{equation}{section}

\titleformat*{\section}{\large \bfseries}
\titleformat*{\subsection}{\normalsize \bfseries}
\titleformat*{\subsubsection}{\small \bfseries}

\theoremstyle{definition}
\newtheorem{theorem}{Theorem}
\newtheorem{definition}{Definition}

\newtheorem{condition}{Condition}
\newtheorem{lemma}{Lemma}
\newtheorem{example}{Example}

\newtheorem{remark}{Remark}
\newtheorem{corollary}{Corollary}

\newcommand\norm[1]{\left\lVert#1\right\rVert}

    \newcommand{\dto}{\xrightarrow{d}}
    \newcommand{\wto}{\xrightarrow{w}}
    \newcommand{\vto}{\xrightarrow{v}}
    \newcommand{\fidi}{\xrightarrow{\text{fidi}}}

    \newcommand{\rmd}{\mathrm{d}}

\usepackage{mathtools}

\begin{document}
\pagenumbering{roman}

\title{ \Large \textbf{Weak Convergence for Self-Normalized Partial Sum Processes in the Skorokhod $M_1$ Topology with Applications to  Regularly Varying Time Series}}

\author{\textbf{Christis Katsouris}\thanks{Dr. Christis Katsouris received his PhD degree from the School of Economic, Social and Political Sciences, University of Southampton, Southampton  SO17 1BJ, UK. Email Address: \color{blue}{\texttt{christiskatsouris@gmail.com}} } \\ First Draft on arXiv: 03 May 2024 
}

\maketitle
\vspace*{-1.1 em}
\begin{abstract}
In this paper we study the weak convergence of self-normalized partial sum processes in the Skorokhod $M_1$ topology for sequences of random variables which exhibit clustering of large values of the same sign. We show that for stationary regularly varying sequences with such properties, their corresponding properly centered self-normalized partial sums processes converge to a stable L\'evy process. The convergence is established in the space of c\'adl\'ag functions endowed with Skorohod's $M_1$ topology, which is more suitable especially for cases in which the standard $J_1$ topology fails to induce weak convergence of joint stochastic functionals. 
\\


\medskip


\end{abstract}

\newpage 

\begin{small}

\begin{spacing}{1.15}
\tableofcontents
\end{spacing}

\end{small}

\newpage 

\setcounter{page}{1}
\pagenumbering{arabic}

\newpage 

\section{Introduction}
\label{Section1}

Weak convergence theory in the space of bounded functions (see, \cite{kuelbs1968invariance},  \cite{andersen1987central} \cite{hoffmann1991stochastic}, \cite{arcones1996some} and \cite{peterson2011convergence}) is commonly used both in statistical as well as in econometric environments with time series regressions\footnote{Notice that time series sequences are stochastic processes models which are classified into two types: one is for continuous type processes such as diffusion processes and others are for jump processes such as point processes.}; especially when asymptotic theory of model estimators and associated test statistics has to be developed. Moreover, the literature on the weak convergence of self-normalized processes (see, \cite{csorgo2003donsker}) - that belong to the class of diffusion processes based on the standard $J_1$ topology, covers joint functional convergence of functionals for both stationary (see, \cite{kulperger2005high}) as well as nonstationary stochastic processes (see, \cite{buchmann2007asymptotic}). However, there are cases in which weak convergence fails to hold (see, \cite{bucher2014uniform}), such as empirical copula and tail dependence processes that have corresponding limit processes with discontinuous trajectories. In this study, we establish the weak invariance principle for self-normalized partial sum processes of stationary sequences based on the Skorokhod $M_1$ topology (see, \cite{basrak2010functional}, \cite{krizmanic2018joint, krizmanic2020joint, krizmanic2022functional}) which requires some special attention. In this more general setting, weak convergence equipped with the $M_1$ topology covers cases where the $J_1$ topology is not applicable due to discontinuities in trajectories of limit processes and thus permits to establish related asymptotic theory.     

Our fundamental framework for developing the asymptotic theory of self-normalized partial sum processes of heavy tailed time series is based on the main results presented in the seminal study of \cite{basrak2010functional}, BKS, who establish weak convergence in the $M_1$ topology for partial sum process. Specifically, the authors show that the convergence takes place in the space of c\'adl\'ag functions endowed with Skorokhod's $M_1$ topology, while the usual $J_1$ topology is not suitable space to ensure convergence as the partial sum processes may exhibit rapid successions of jumps within temporal clusters of large values, collapsing in the limit to a single jump. However, BKS only consider the case of partial sums and partial maxima sums and there is no related limit theory for self-normalized partial sum processes of stationary regularly varying processes. In this paper, we focus on theoretical and theoretical applications for exactly this case. 

Various studies in the literature focus on developing asymptotic theory for self-normalized sums\footnote{Weak convergence of self-normalized partial-sum processes is also employed in the case of nonstationary time series processes to establish asymptotic approximations. However in this study we only consider strictly stationary sequences with relevant mixing conditions. } of independent random variables under certain regularity conditions (see, \cite{shao1997self}, \cite{shao1999cramer} and \cite{fan2019self}) such as a martingale sequence condition on innovation sequences. More precisely, the asymptotic behaviour of self-normalized estimators and test statistics which have a martingale representation can be obtained using principles such as Cram\'er-type moderate deviations and Berry-Esseen bounds (see, \cite{bentkus1996berry}) and \cite{gao2022refined} who extend to the case of dependent random variables. Fewer studies exist specifically for self-normalized partial sums of heavy tailed processes, with a recent approach presented by \cite{matsui2023self}, although in this framework the authors  do not consider the condition of small vanishing values.

\newpage

\subsection{Preliminary Theory}

The development of invariance principles, are based on a well-defined metric space, such as $d(.,.)$ be the Skorohod metric on $D(0, \infty)$. Since local uniform convergence implies Skorohod convergence, we get that $d \big( X_0^{( \varepsilon) } (\,\cdot\,),  X_0(\,\cdot\,) \big) \to 0$  almost surely as $\varepsilon \to 0$, and hence, since almost sure convergence implies weak convergence, $X_0^{( \varepsilon) } (\,\cdot\,) \Rightarrow X_0(\,\cdot\,)$. 

Consider $\left( X_i \right)_{i \geq 1}$ to denote a sequence of independent random variables with zero means and finite variances. Define with $S_n = \sum_{i=1}^n X_i$, $B_n^2 = \sum_{i=1}^n \mathbb{E} X_i^2$ and $V_n^2 =  \sum_{i=1}^n X_i^2$. Then, the self-normalized sum $S_n / V_n \to \mathcal{N}(0,1)$ if and only if the random variable $X$, is in the domain of attraction of the normal law. In addition, if the self-normalized sums $S_n / V_n, n \in \mathbb{N}$, are stochastically bounded, then they are uniformly sub-Gaussian in the sense that
\begin{align}
\underset{ n \in \mathbb{N} }{ \text{sup} } \ \mathbb{E} \left[ e^{t S_n / V_n }  \right] \leq 2 e^{ct^2} \ \ \text{for all} \ \ t \in \mathbb{R} \ \ \text{and some} \ c < \infty. 
\end{align}  
The normal law property holds when considering the joint weak convergence of partial-sum processes. Let $X_{1,n},...,X_{n,n}, n \geq 1$, be a triangular array of independent infinitesimal random variables. Then the joint functional $( S_n, V_n^2 )$ satisfies the following weak convergence result 
\begin{align}
\left(  \sum_{i=1}^n X_{i,n},  \sum_{i=1}^n X_{n,n}^2 \right) \to_D (U,V), 
\end{align}
for a nondegenerate pair $(U,V)$.

Moreover, it holds that $\sum_{i=1}^n X_{i,n} \big/ \sqrt{ \sum_{i=1}^n X_{i,n}^2 }  \to_{D} \mathcal{N} (0,1)$ \textit{iff} for some $\kappa > 0$, 
\begin{align}
\sum_{i=1}^n X_{i,n} \overset{ d }{ \to } \kappa \mathcal{N}(0,1) \ \ \text{and} \ \  \sum_{i=1}^n X_{i,n}^2 \overset{ p }{ \to } \kappa^2. 
\end{align} 
An extension of the self-normalized central limit theorem to Donsker type functional central limit theorem applies. Define the partial-sum process $S_{[nt]} = \sum_{i=1}^{[nt]} X_i$. The concept of invariance principles was first introduced by the paper of \cite{erdos1946certain} and \cite{donsker1951invariance} as well as generalized by \cite{prokhorov1956convergence} (see, discussion in \cite{kuelbs1968invariance}). Furthermore, \cite{kuelbs1968invariance} consider the weak convergence of these random elements into continuous functionals on the space of real-valued functions that includes Gaussian measures as analogues to Wiener measures. The particular property of such topological spaces  motivated the construction of a new metric such as the approach of \cite{skorokhod1956limit}, who introduced a number of metrics on the space of c\'adlag functions on $[0,1]$, with $J_1$ topology being the most widely used in statistics and econometric asymptotic theory. 

Suppose we consider the asymptotic behaviour of a self-normalized partial-sum process,  as $n \to \infty$ it holds that $S_{[ n t ]} / V_n \to_D W(t)$ on $\left( \mathcal{D}[0,1], \rho \right)$, where $\rho$ is the sup-norm metric for functions in $\mathcal{D}[0,1]$, and $\left\{ W(t), 0 \leq t \leq 1 \right\}$ is a standard Wiener process. The limit result below of \cite{csorgo2003donsker} shows that the behaviour in probability of the self-normalized process $\big\{ S_{\floor{nt} } / V_n \big\}_{ t \in [0,1] }$ coincides with the behaviour of a standard Brownian motion $W = \left\{ W(t) \right\}_{ t \in [0,1] }$.

\newpage

On an appropriate probability space it holds that 
\begin{align}
\underset{ t \in [0,1] }{ \text{sup} } \left| \frac{S_{\floor{nt} }}{V_n} - \frac{W(nt)}{ \sqrt{n} } \right| = o_p(1). 
\end{align} 
In this article, we consider the asymptotics for self-normalized partial-sum processes where joint weak convergence arguments of functionals to the $J_1$ topology do not directly apply. The reason is that this topology is restricted to the space of functions that are continuous when points are approached from a certain direction (see, \cite{bucher2014uniform}). More precisely, two functions are close in the $J_1$ topology only if, up to a small perturbation in the coordinates, their supremum distance is small. This implies that their jumps, if any, must match both in location and in magnitude. Since, without loss of generality a sequence of continuous functions cannot converge to a discontinuous function, we cannot use arguments such as the hypiconvergence.

We establish weak convergence of the joint functional $L_n := ( L_{1n}, L_{2n} )$ based on weak convergence arguments of the corresponding point process of jumps, for partial sum processes of stationary regularly varying sequences in the $M_1$ topology (see, \cite{ledger2016skorokhod}). Using the connection between L\'evy processes and infinitely divisible distributions as discussed in the study of \cite{upadhye2022unified} and \cite{krizmanic2018joint} (see, also \cite{jacod2011discretization}), we demonstrate that these convergence results hold for cases in which the tail index satisfies $\alpha \in (0,2)$. Our main result  in relation to the joint weak convergence is presented in our theorems which show that the particular topological convergence holds for a more general dependence structure than the standard \textit{i.i.d} setting regardless of the number of observations in each cluster.

The rest of the paper is organized as follows.  Section \ref{Section2} discusses the relevant preliminary theory related to regular variation and the forms of weak and strong $M_1$ topologies. Section \ref{Section3} presents the main results of our framework on weak convergence of weak convergence of self-normalized partial-sum processes for the cases in which $\alpha \in (0,1)$ and $\alpha \in [1,2)$. Section \ref{Section4} illustrates theoretical applications of the main results to the construction of self-normalized partial-sum processes for stationary regularly varying sequences. All main proofs can be found in the Appendix of the paper.

\section{Stationary regularly varying sequences}
\label{Section2}

In this section we present related preliminary theory which apply specifically to stationary regularly varying sequences, following \cite{krizmanic2014weak, krizmanic2016functional, krizmanic2017weak} and \cite{basrak2015multivariate}.

\subsection{Regular variation}

Let $\mathbb{E}^{d}=[-\infty, \infty]^{d} \setminus \{ 0 \}$. We equip
$\mathbb{E}^{d}$ with the topology in which a set $B \subset \mathbb{E}^{d}$ has compact closure \textit{if and only if} it is bounded away from zero, if there exists $u > 0$ such that $
B \subset \mathbb{E}^{d}_u = \{ x \in \mathbb{E}^{d} : \|x\| >u \}$. Here $\| \cdot \|$ denotes the max-norm on $\mathbb{R}^{d}$, i.e.\
$\displaystyle \| x \|=\max \{ |x_{i}| : i=1, \ldots , d\}$ where
$x=(x_{1}, \ldots, x_{d}) \in \mathbb{R}^{d}$. Denote by $C_{K}^{+}(\mathbb{E}^{d})$ the class of all non-negative, continuous functions on $\mathbb{E}^{d}$ with compact support.

\newpage

\begin{definition}
A stationary process $(X_{n})_{n \in \mathbb{Z}}$ is \emph{(jointly) regularly varying} with index $\alpha \in (0,\infty)$ if for any nonnegative integer $k$ the
$kd$-dimensional random vector $X = (X_{1}, \ldots , X_{k})$ is multivariate regularly varying with index $\alpha$, i.e.\ there exists a random vector $\Theta$ on the unit sphere $\mathbb{S}^{kd-1} = \{ x \in \mathbb{R}^{kd} : \|x\|=1 \}$ such
that for every $u \in (0,\infty)$ and as $x \to \infty$,
\begin{equation}\label{e:regvar1}
\frac{ \mathbb{P} (\|X\| > ux,\,X / \| X \| \in \cdot \, )}{ \mathbb{P}(\| X \| >x)}
\wto u^{-\alpha} \mathbb{P} ( \Theta \in \cdot \,),
\end{equation}
\end{definition}

Then, regular variation can be expressed in terms of vague convergence of
measures on $\mathbb{E}$
\begin{equation}
\label{e:onedimregvar}
n \mathbb{P} ( a_n^{-1} X_i \in \cdot \, ) \vto \mu( \, \cdot \,),
\end{equation}
where $(a_{n})_{ n \in \mathbb{N} }$ is a sequence of positive real numbers such that the following condition holds
\begin{equation}
\label{e:niz}
n \mathbb{P} ( |X_{1}| > a_{n}) \to 1
\end{equation}
as $n \to \infty$, and $\mu$ is a nonzero Radon measure on $\mathbb{E}$ given by
\begin{equation}
\label{e:mu}
\mu(\rmd x) = \bigl( p \, 1_{(0, \infty)}(x) + q \, 1_{(-\infty, 0)}(x) \bigr) \, \alpha |x|^{-\alpha-1}\,\rmd x,
\end{equation}
for some $p \in [0,1]$, with $q=1-p$.

Specifically, Theorem 2.1 in \cite{basrak2009regularly}  provides a convenient characterization of joint regular variation:~it is necessary and sufficient that there exists a process $(Y_n)_{n \in \mathbb{Z}}$
with $\Pr(|Y_0| > y) = y^{-\alpha}$ for $y \geq 1$ such that, as $x \to \infty$,
\begin{equation}
\label{e:tailprocess}
\bigl( (x^{-1}\ X_n)_{n \in \mathbb{Z}} \, \big| \, | X_0| > x \bigr)
\fidi (Y_n)_{n \in \mathbb{Z}},
\end{equation}
where "$\fidi$" denotes convergence of finite-dimensional
distributions. The process $(Y_{n})$ is called
the \emph{tail process} of $(X_{n})_{ n \in \mathbb{N} }$.

\subsection{Point processes and dependence conditions}
\label{s:pp}

Let $(X_{n})_{ n \in \mathbb{N} }$ be a stationary sequence of random variables and assume it is jointly regularly varying with index $\alpha >0$. Let $(Y_{n})_{ n \in \mathbb{Z} }$ be the tail process of $(X_{n})_{ n \in \mathbb{Z} }$. Suppose that the stochastic process $N_n$ is defined as below
\begin{equation*}
\label{E:ppspacetime}
N_{n} = \sum_{i=1}^{n} \delta_{(i / n,\,X_{i} / a_{n})} \qquad \textrm{for all} \ n\in \mathbb{N},
\end{equation*}
where the norming sequence $\left\{ a_{n} \right\}$ is defined such that the condition of expression (\ref{e:niz}) holds. 

The point process convergence for the sequence $(N_{n})$ has been established by \cite{basrak2010functional} on the space $[0,1] \times \mathbb{E}_{u}$ for any threshold $u>0$, with the limit depending on the threshold. A useful convergence result for $N_{n}$ without the restrictions to various domains was obtained by \cite{basrak2016complete}.  Appropriate weak dependence conditions for this convergence result are given below. Both Condition \ref{cond1} and Condition \ref{c:mixcond2} provide necessary regularity assumptions.

\newpage

\begin{condition}
\label{cond1}
There exists a sequence of positive integers $(r_{n})$ such that $r_{n} \to \infty $ and $r_{n} / n \to 0$ as $n \to \infty$ and such that for every $f \in C_{K}^{+}([0,1] \times \mathbb{E})$, denoting $k_{n} = \lfloor n / r_{n} \rfloor$, as $n \to \infty$,
\begin{equation}
\label{e:mixcon}
\mathbf{E} \biggl[ \exp \biggl\{ - \sum_{i=1}^{n} f \biggl(\frac{i}{n}, \frac{X_{i}}{a_{n}}
\biggr) \biggr\} \biggr] - \prod_{k=1}^{k_{n}} \mathbf{E} \biggl[ \exp \biggl\{ - \sum_{i=1}^{r_{n}} f \biggl(\frac{kr_{n}}{n}, \frac{X_{i}}{a_{n}} \biggr) \biggr\} \biggr] \to 0.
\end{equation}
\end{condition}

\begin{condition}
\label{c:mixcond2}
There exists a sequence of positive integers $(r_{n})$ such that $r_{n} \to \infty $ and $r_{n} / n \to 0$ as $n \to \infty$ and such that for every $u > 0$,
\begin{equation}
\label{e:anticluster}
\lim_{m \to \infty} \limsup_{n \to \infty}
\mathbb{P} \biggl( \max_{m \leq |i| \leq r_{n}} | X_{i} | > ua_{n}\,\bigg|\,| X_{0}|>ua_{n} \biggr) = 0.
\end{equation}
\end{condition}

\begin{remark}
Condition~\ref{cond1} is implied by the strong mixing property (see \cite{krizmanic2016functional}), with the sequence $(\xi_{n})_{ n \in \mathbb{N} }$ being a strongly mixing sequence if $\alpha (n) \to 0$ as $n \to \infty$, where
\begin{align*}
\alpha (n) = \sup \bigg\{  \big| \mathbb{P} (A \cap B) - \mathbb{P}(A) \mathbb{P}(B) \big| : A \in \mathcal{F}_{-\infty}^{0}, B \in \mathcal{F}_{n}^{\infty} \bigg\}    
\end{align*}
and $\mathcal{F}_{k}^{l} = \sigma( \{ \xi_{i} : k \leq i \leq l \} )$ for $-\infty \leq k \leq l \leq \infty$  such that $\alpha (n)$ represents the mixing coefficient.
\end{remark}

For a detailed discussion on joint regular variation and the dependence properties of Condition~\ref{cond1} and Condition \ref{c:mixcond2} we refer to Section 3.4 in \cite{basrak2010functional}. These conditions are applicable to various time series regression settings which include among others, moving averages processes\footnote{Further applications presented in the literature which could be extended based on the regularity conditions we provide here include ARMAX processes (see, \cite{ferreira2013extremes}) as well as residual marked empirical processes of the form $\alpha_n(x) = \sum_{i=1}^n \mathsf{g} \left( X_{i-1}/ a_n \right) \left( \mathbf{1} \left\{ \epsilon_i \leq x \right\} - F(x) \right)$ (see, \cite{chan2013marked}), which are used for constructing goodness-of-fit tests.}, stochastic volatility models as well as GARCH-type processes (see, Section 4 in \cite{basrak2010functional}). Notice that Condition~\ref{cond1} presented above is as Condition 2.2. in \cite{krizmanic2014weak}, Condition 2.1 in \cite{krizmanic2016functional} and Condition 2.1 in \cite{krizmanic2017weak}.

\medskip

In order to present the main weak convergence results for the pair $( L_{1n}, L_{2n} )$ we need to define some relevant terms used in our analysis of random regularly varying sequences. By Proposition 4.2 in \cite{basrak2016complete}, under Condition~\ref{c:mixcond2} the following result holds
\begin{eqnarray}
\label{E:theta:spectral}
\theta := 
\mathbb{P} \left( \underset{ i \geq 1 }{ \mathsf{sup} }| Y_{i}| \le 1 \right) 
= 
\mathbb{P} \left( \underset{ i \leq - 1 }{ \mathsf{sup} }  | Y_{i}| \le 1 \right) > 0,
\end{eqnarray}
where $\theta$ denotes the extremal index of the univariate sequence $(| X_{n} |)_{ n \in \mathbb{N} }$. 

Under the assumption of joint regular variation and Condition~\ref{cond1} and \ref{c:mixcond2}, by Theorem 3.1 in Basrak and Tafro~\cite{basrak2016complete}, the following convergence in distribution holds, 
\begin{equation}
\label{e:BaTa}
N_{n} \dto N := \sum_{i}\sum_{j}\delta_{(T_{i}, P_{i}\eta_{ij})} \in [0,1] \times \mathbb{E} \ \ \ \text{as} \ n \to \infty, 
\end{equation}
where $\sum_{i=1}^{\infty}\delta_{(T_{i}, P_{i})}$ is a Poisson process on $[0,1] \times (0,\infty)$ with intensity measure $Leb \times \nu$.

\newpage

Moreover, $\nu(\rmd x) = \theta \alpha x^{-\alpha-1} \boldsymbol{1}_{(0,\infty)}(x)\,\rmd x$ and $(\sum_{j=1}^{\infty} \delta_{\eta_{ij}})_{i}$ is an i.i.d.~sequence of point processes in $\mathbb{E}$ independent of $\sum_{i} \delta_{(T_{i}, P_{i})}$ with common distribution equal to the distribution of $\sum_{j}\delta_{Z_{j} / L_{Z} }$, where $L_{Z}= \sup_{j \in \mathbb{Z}}|Z_{j}|$ and $\sum_{j}\delta_{Z_{j}}$ is distributed as $( \sum_{j \in \mathbb{Z}} \delta_{Y_j} \,|\, \sup_{i \le -1} | Y_i| \le 1)$. 

\begin{corollary}(\cite{davis1995point})
\label{Corollary1}
Let $\sum_{j}\delta_{\eta_{j}}$ be a point process with a limit distribution equal to the distribution of $\sum_{j}\delta_{\eta_{1j}}$. Then, for $ \alpha \leq 1$ it holds that
\begin{equation}
\label{e:MW1}
\mathbf{E} \Big( \sum_{j}|\eta_{j}| \Big)^{\alpha} < \infty.
\end{equation}
\end{corollary}
\begin{remark}
The finite moment condition given by Corollary \ref{Corollary1} provides a necessary condition for the weak convergence of point processes into a stable distribution when $\alpha \leq 1$. Although  the particular finiteness of moments condition of point processes might not hold for $\alpha > 1$ (see \cite{mikosch2014cluster}). 
\end{remark}

\subsection{The weak and strong $M_{1}$ topologies}
\label{ss:j1m1}

In this section, we introduce the definition of the Skorokhod weak $M_{1}$ topology in a general space $D([0,1], \mathbb{R}^{d})$ of $\mathbb{R}^{d}$--valued c\`{a}dl\`{a}g functions on the closed set $[0,1]$. The definitions presented below  can be found in Section 3 of \cite{basrak2015multivariate} and in Section 2.3 of  \cite{krizmanic2017weak}.

\begin{definition}[Completed graph] 
For $x \in D([0,1], \mathbb{R}^{d})$ the completed (thick) graph of $x$ is the set 
\begin{align}
G_{x}  = \big\{ (t,z) \in [0,1] \times \mathbb{R}^{d} : z \in [[x(t-), x(t)]] \big\},    
\end{align}
where $x(t-)$ is the left limit of $x$ at $t$ and $[[a,b]]$ is the product segment, such that \begin{align}
[[a,b]]=[a_{1},b_{1}] \times [a_{2},b_{2}] \ldots \times [a_{d},b_{d}]   
\end{align}
where $a=(a_{1}, a_{2}, \ldots, a_{d})$ and $b=(b_{1}, b_{2}, \ldots, b_{d}) \in
\mathbb{R}^{d}$.
\end{definition}

\begin{definition}[Order of graph]
An order on the graph $G_{x}$ is defined by assuming that $(t_{1},z_{1}) \le
(t_{2},z_{2})$ if either (i) $t_{1} < t_{2}$ or (ii) $t_{1} = t_{2}$
and $|x_{j}(t_{1}-) - z_{1j}| \le |x_{j}(t_{2}-) - z_{2j}|$, which holds 
for all $j=1, 2, \ldots, d$. 
\end{definition}

\begin{definition}[Partial Order of graph]
The relation $\preceq$ induces only a partial order on the graph $G_{x}$. A weak parametric representation of the graph $G_{x}$ is a continuous non-decreasing function $(r,u)$ mapping $[0,1]$ into $G_{x}$, with $r \in C([0,1],[0,1])$ being the time component and $u \in C([0,1], \mathbb{R}^{d})$ being the spatial component, such that $r(0)=0, r(1)=1$ and $u(1)=x(1)$.
\end{definition}

\begin{definition}
Let $\Pi_{w}(x)$ denote the set of weak parametric representations of the graph $G_{x}$. For $x_{1},x_{2} \in D([0,1], \mathbb{R}^{d})$ define
\begin{align}
d_{w}(x_{1},x_{2})
= \inf \big\{ \|r_{1}-r_{2}\|_{[0,1]} \vee \|u_{1}-u_{2}\|_{[0,1]} : (r_{i},u_{i}) \in \Pi_{w}(x_{i}), i=1,2 \big\},
\end{align}
where $\|x\|_{[0,1]} = \sup \{ \|x(t)\| : t \in [0,1] \}$. 
\end{definition}

\newpage

\begin{definition}(Weak $M_{1}$ topology)
We say that $x_{n} \to x$ in $D([0,1], \mathbb{R}^{d})$ for a sequence
$(x_{n})_{ n \in \mathbb{R} }$ in the weak Skorokhod $M_{1}$ topology if $d_{w}(x_{n},x)\to 0$ as $n \to \infty$.
\end{definition}

\begin{definition}(Strong $M_{1}$ topology)
For $x \in D([0,1], \mathbb{R}^{d})$ we let
\begin{align}
\Gamma_{x}  = \big\{ (t,z) \in [0,1] \times \mathbb{R}^{d} : z \in [x(t-), x(t)] \big\},
\end{align}
where $[a,b] = \{  \lambda a + (1-\lambda)b : 0 \leq \lambda \leq 1 \}$ for $a, b \in \mathbb{R}^{d}$. We say that $(r,u)$ is a strong parametric representation of $\Gamma_{x}$ if it is a continuous non-decreasing function mapping $[0,1]$ onto $\Gamma_{x}$. 

\medskip

Denote by $\Pi(x)$ the set of all strong parametric representations of the completed (thin) graph $\Gamma_{x}$. Then for $x_{1},x_{2} \in D([0,1], \mathbb{R}^{d})$ denote with
\begin{align}
d_{M_{1}}(x_{1},x_{2})
= \inf \big\{ \|r_{1}-r_{2}\|_{[0,1]} \vee \|u_{1}-u_{2}\|_{[0,1]} : (r_{i},u_{i}) \in \Pi(x_{i}), i=1,2 \big\}.
\end{align}
Then, we say that the distance function $d_{M_{1}}(x_{1},x_{2})$ is a metric on $D([0,1], \mathbb{R}^{d})$, and the induced topology is called the (standard or strong) Skorokhod $M_{1}$ topology.
\end{definition}

\begin{remark}
Notice that the $WM_{1}$ topology is weaker than the standard $M_{1}$ topology on $D([0,1], \mathbb{R}^{d})$, but they coincide in the case that $d=1$. Furthermore, the $WM_{1}$ topology coincides with the topology induced by the metric
\begin{equation}
\label{e:defdp}
 d_{p}(x_{1},x_{2})= \max \big\{ d_{M_{1}}(x_{1j},x_{2j}) : j=1,2, \ldots, d \big\}
\end{equation}
for $x_{i}=(x_{i1}, x_{i2}, \ldots, x_{id}) \in D([0,1], \mathbb{R}^{d})$ and $i=1,2$. Then, we say that the metric $d_{p}(x_{1},x_{2})$ induces the product topology on $D([0,1], \mathbb{R}^{d})$. For detailed discussion of the strong and weak $M_{1}$ topologies we refer to sections 12.3--12.5 in \cite{Whitt2002stochastic}.
\end{remark}

Next, we use the above definitions in order to introduce the induced measurability results. Denote by $C^{\uparrow}_{0}([0,1], \mathbb{R})$ the subset of functions $x$ in the space $D ( [0,1], \mathbb{R})$ that are continuous and nondecreasing which satisfy $x(0)>0$. Therefore, we assume that $C^{\uparrow}_{0}([0,1], \mathbb{R})$ is a measurable subset of $D( [0,1], \mathbb{R} )$ with the Borel $\sigma$--fields associated with the Skorokhod $J_{1}$ and $M_{1}$ topologies, which coincide with the Kolmogorov $\sigma$--field generated by the projection maps (see Theorem 11.5.2 and Theorem 11.5.3 in \cite{Whitt2002stochastic}).

\begin{remark}
We only consider weak convergence of partial-sum processes in the subset $C^{\uparrow}_{n}([0,1], \mathbb{R})$. Notice that the $J_1$ topology corresponds to a closed set while the standard $M_1$ topology corresponds to an open set, which it is still assumed to be a measurable set. As a result, the set $C^{\uparrow}_{n}([0,1], \mathbb{R})$ corresponds to the subset of c\`{a}dl\`{a}g functions which corresponds to continuous, positive and non-decreasing functions, which ensures the measurability conditions. 
\end{remark}

\newpage 

\begin{lemma}
\label{Lemma1}
The set $C^{\uparrow}_{0}([0,1], \mathbb{R})$ is a measurable subset of $D[0,1], \mathbb{R})$ with the $M_{1}$ topology.
\end{lemma}

\begin{proof}
The proof of \ref{Lemma1} can be found in the Appendix.    
\end{proof}

Consequently, to facilitate the development of limit theory for the proposed framework, we introduce the concept of $M_{1}$ continuity of multiplication and division of two c\`{a}dl\`{a}g functions which are well-established results in the literature. Specifically, the first result presented in Lemma \ref{l:contmultpl} is based on Theorem 13.3.2 of  \cite{Whitt2002stochastic}, while the second result presented in Lemma \ref{l:M1div} is based on the fact that for monotone functions $M_{1}$ convergence is equivalent to point-wise convergence in a dense subset of $[0,1]$ including $0$ and $1$ (see \cite{Whitt2002stochastic}, Corollary 12.5.1). 

Denote by $\textrm{Disc}(x)$ the set of discontinuity points of $x \in D([0,1], \mathbb{R})$.

\begin{lemma}
\label{l:contmultpl}
Suppose that $x_{n} \to x$ and $y_{n} \to y$ in $D([0,1], \mathbb{R})$ with the $M_{1}$ topology. If it holds that for each $t \in \textrm{Disc}(x) \cap \textrm{Disc}(y)$, $x(t)$, $x(t-)$, $y(t)$ and $y(t-)$ are all nonnegative and $[x(t)-x(t-)][y(t)-y(t-)] \geq 0$, then 
\begin{align}
x_{n}y_{n} \to xy, \ \ \  \in D([0,1], \mathbb{R}) 
\end{align}
with the $M_{1}$ topology, where $(xy)(t) = x(t)y(t)$ for $t \in [0,1]$.
\end{lemma}

\begin{lemma}
\label{l:M1div}
Consider the function $h \colon D([0,1], \mathbb{R}) \times C^{\uparrow}_{0}([0,1], \mathbb{R}) \to D([0,1], \mathbb{R})$ defined by $h(x,y)= \frac{x}{y}$, where
\begin{align}
 \Big(\frac{x}{y} \Big)(t) = \frac{x(t)}{y(t)}, \qquad t \in [0,1].    
\end{align}
Then the function h is continuous when $D([0,1], \mathbb{R}) \times C^{\uparrow}_{0}([0,1], \mathbb{R})$ is endowed with the weak $M_{1}$ topology and $D([0,1], \mathbb{R})$ is endowed with the standard $M_{1}$ topology.
\end{lemma}

\begin{proof}
The proof of Lemma \ref{l:M1div} can be found in the Appendix.    
\end{proof}

The next section considers conditions under which the joint partial sum process $L_n$ satisfies a functional limit theorem with the limit $L = \left( L_1, L_2 \right)$, where $L_1$ is an $\alpha-$stable L\'{e}vy process and $L_2$ is is an $\alpha/2-$stable L\'{e}vy process. Specifically, we express the components of the  limiting process $L = \left( L_1, L_2 \right)$ as functionals of the limiting point process $N = \sum_i \sum_j \delta_{( T_i, P_i \eta_{ij} )}$ based on the characteristic triple of $L_1$, given by $( a, \nu, b )$. Notice that $a \in \mathbb{R}$ is the location parameter and $\alpha \in (0,2)$ is the stability parameter, which is useful in determining the decay of the tail of the distribution of $X$. Specifically, we consider cases such that $\alpha < 2$, which excludes Gaussian processes. We introduce a mode of convergence for real-valued, locally bounded functions on a locally compact metrizable space which does not necessarily require that the \textit{i.i.d} assumption holds, thereby allowing for clustering of dependent observations for stationary regularly varying sequences. Our main result establishes the weak convergence of self-normalized partial sum processes of stationary regularly varying sequences to a stable L\'{e}vy process in the space $D[0,1]$ equipped with Skorokhod's $M_1$ topology. 

\newpage

\section{Joint Functional Convergence of Partial Sum Processes}
\label{Section3}

\subsection{Joint Functional Convergence of Partial Sum Processes with $\alpha \in (0,1)$} 

Let $(X_{n})_{ n \in \mathbb{Z} }$ be a stationary sequence of random variables, jointly regularly varying with index $\alpha \in (0,1)$. Assume Conditions~\ref{cond1} and \ref{c:mixcond2} hold. Let
\begin{align}
S_{n} = \sum_{k=1}^{n}X_{k}, \qquad V^2_{n} = \sum_{k=1}^{n}X_{k}^{2}, \qquad n \in \mathbb{N}     
\end{align}
and denote with $L_{n}(\cdot)$, the joint partial sum process where $ L_n(\cdot) := \big( L_{1n}(\cdot), L_{2n}(\cdot) \big) $ such that 
\begin{align}
L_{n}(t) =  \left(  \sum_{k=1}^{[nt]} \frac{X_{k}}{a_{n}}, \sum_{k=1}^{ \lfloor nt \rfloor} \frac{X_{k}^{2}}{a_{n}^{2}}  \right),  \qquad t \in [0,1],   
\end{align}
where the sequence $( a_{n} )_{n \in \mathbb{N} }$ defined by expression \eqref{e:onedimregvar} such that property (\ref{e:niz}) holds. 

The first coordinate of the joint functional $L_n(.)$ corresponds to partial sums of a linear process while the second coordinate corresponds to partial sums of squares of a linear process. The key idea to establish weak convergence results on a probability space $\left( \Omega, \mathbb{P}, \mathcal{F} \right)$, as $n \to \infty$, is to represent the joint stochastic process $L_n$ as the image of the time-space point process $N_n$ under an appropriate summation functional. Using the continuity properties of $L_n$ and by an application of the continuous mapping theorem we can transfer the weak convergence of the time-space point process  $N_n$, given by (\ref{e:BaTa}) to weak convergence of the joint stochastic process $L_n$ in the $M_1$ topology. Fix $0 < u < \infty$ and define the summation functional
\begin{align}
&\Phi^{(u)} \colon \mathbf{M}_{p}([0,1] \times \mathbb{E}) \to D([0,1], \mathbb{R}^{2})  
\\
&\Phi^{(u)} \Big( \sum_{i}\delta_{(t_{i}, x_{i})} \Big) (t)
=  
\Big( \sum_{t_{i} \leq t}x_{i}\,1_{\{u < |x_{i}| < \infty \}},  \sum_{t_{i} \leq t} x_{i}^{2}\,1_{\{u < |x_{i}| < \infty \}}  \Big), \qquad t \in [0,1],   
\end{align}
where the space $\mathbf{M}_p([0,1] \times \mathbb{E})$ of Radon point measures on $[0,1] \times \mathbb{E}$ is equipped with the vague topology (see \cite{resnick2008extreme}, Chapter 3). Let $\Lambda = \Lambda_{1} \cap \Lambda_{2}$, where
\begin{multline}
\label{lampdas1}
\Lambda_{1} =
\{ \eta \in \mathbf{M}_{p}([0,1] \times \mathbb{E}) :
\eta ( \{0,1 \} \times \mathbb{E}) = 0 = \eta ([0,1] \times \{ \pm \infty, \pm u \}) \}, 
\\[1em]
\shoveleft 
\Lambda_{2} =
\{ \eta \in \mathbf{M}_{p}([0,1] \times \mathbb{E}) :
\eta ( \{ t \} \times (u, \infty]) \cdot \eta ( \{ t \} \times [-\infty,-u)) = 0, \  \text{for all $t \in [0,1]$} \}.
\end{multline}
The definitions of the sets $\Lambda_1$ and $\Lambda_2$ are instrumental for our main results. The elements of $\Lambda_2$ have the property that atoms in $[0,1] \times \mathbb{E}_{u}$ with the same time coordinate are all on the same side of the time axis. Similar to Lemma 3.1 in \cite{basrak2010functional} one can prove that the Poisson cluster process $N$ given by (\ref{e:BaTa}) a.s.~belongs to the set $\Lambda$ given that each cluster of large  values of $(X_{n})_{ n \in \mathbb{N} }$ contains only values of the same sign. The link between the tail measure of a regularly varying time series and its spectral tail process is studied by \cite{dombry2018tail}.

\newpage

\begin{lemma}
\label{l:prob1}
Assume that with probability one the tail process $(Y_{i})_{i \in \mathbb{Z}}$ of the sequence $(X_{n})_{ n \in \mathbb{N} }$ has no two values of the opposite sign. Then it holds that,
\begin{align}
\mathbb{P} ( N \in \Lambda ) = 1.    
\end{align}
\end{lemma}
Next, we examine the continuity properties of the stochastic process $\Phi^{(u)}$ is over the set $\Lambda$, where $\Lambda = \Lambda_{1} \cap \Lambda_{2}$ and $\Lambda_{1}, \Lambda_{2}$ are given by \eqref{lampdas}.
\begin{lemma}
\label{l:contfunct}
The summation functional $\Phi^{(u)} \colon \mathbf{M}_{p}([0,1] \times \mathbb{E}) \to D([0,1], \mathbb{R}^{2})$ is continuous on the set $\Lambda$, when $D([0,1], \mathbb{R}^{2})$ is endowed with the weak $M_{1}$ topology.
\end{lemma}
In the theorem below we establish functional convergence of the stochastic process $L_{n}$, with the limit $L = (L_{1}, L_{2})$ consisting of stable L\'{e}vy processes $L_{1}$ and $L_{2}$. Without loss of generality, the distribution of a L\'{e}vy process $V$ is characterized by its characteristic triple, that is, the characteristic triple of the infinitely divisible distribution of $V(1)$. Then, the characteristic function of $V(1)$ and the characteristic triple $(a, \nu, b)$ are related in the following way:
\begin{equation}
\label{e:Kintchin}
\mathbf{E} \left[ e^{iz V(1) } \right] 
= 
\exp \biggl( -\frac{1}{2}az^{2} + ibz + \int_{\mathbb{R}} \bigl( e^{izx}-1-izx \boldsymbol{1}_{[-1,1]}(x) \bigr)\,\nu(\rmd x) \biggr)
\end{equation}
for some $z \in \mathbb{R}$. Here $a \ge 0$, $b \in \mathbb{R}$ are constants, and $\nu$ is a measure on $\mathbb{R}$ satisfying
\begin{align}
\nu ( \{0\})=0 \qquad \text{and} \qquad \int_{\mathbb{R}}(|x|^{2} \wedge 1)\,\nu(\rmd x) < \infty.    
\end{align}

\begin{theorem}[see, \cite{krizmanic2020joint}]
\label{t:functconvergence}
Let $(X_{n})_{ n \in \mathbb{N} }$ be a stationary sequence of random variables, jointly regularly varying with index $\alpha\in(0,1)$, and of which the tail process $(Y_{i})_{i \in \mathbb{Z}}$ almost surely has no two values of the opposite sign. Suppose that Conditions~\ref{cond1} and~\ref{c:mixcond2} hold. Then it holds that,
\begin{align}
 L_{n} \dto L \qquad \textrm{as} \ n \to \infty,    
\end{align}
in $D([0,1], \mathbb{R}^{2})$ endowed with the weak $M_{1}$ topology, where $L = (L_{1}, L_{2})$, $L_{1}$ is an $\alpha$--stable L\'{e}vy process with characteristic triple $(0,
\nu_{1}, \gamma_{1})$ and $L_{2}$ is an $\alpha/2$--stable L\'{e}vy process with characteristic triple $(0, \nu_{2}, \gamma_{2})$, where
\begin{align}
\nu_{1}(\rmd x) &= \theta \alpha \big( c_{+} 1_{(0,\infty)}(x) + c_{-} 1_{(-\infty, 0)}(x) \big) |x|^{-\alpha -1}\,\rmd x,    
\\
\nu_{2}(\rmd x) &=  \frac{\theta \alpha}{2} \mathbf{E}\bigg( \sum_{j}\eta_{j}^{2}\bigg)^{\alpha/2} x^{-\alpha/2 -1}1_{(0,\infty)}(x) \,\rmd x,
\\
c_{+} &= \mathbf{E} \bigg[ \Big( \sum_{j}\eta_{j} \Big)^{\alpha} 1_{ \{ \sum_{j}\eta_{j} > 0 \}} \bigg], \quad
c_{-} = \mathbf{E} \bigg[ \Big( - \sum_{j}\eta_{j} \Big)^{\alpha} 1_{ \{ \sum_{j}\eta_{j} < 0 \}} \bigg]    
\\
\gamma_{1} &=  \frac{\theta \alpha}{1-\alpha} (c_{+}-c_{-}), \qquad \gamma_{2} = \frac{\theta \alpha}{2-\alpha} \mathbf{E}\bigg( \sum_{j}\eta_{j}^{2}\bigg)^{\alpha/2}.  
\end{align}
\end{theorem}

\newpage

\begin{remark}
Recall that a Levy process $V$ with the characteristic triple given by $(  b, c, \nu )$,  with respect to a truncation function $\kappa$, is a stochastic process with a characteristic function given by 
\begin{align}
\mathbb{E} \big[  e^{ i z L_t } \big] = \mathsf{exp} \left(  - \frac{1}{2} t c  z^2 - i t b z + t \int_{ \mathbb{R} } \big(  e^{ i z x } - 1 - i z \kappa(x) \big) \nu( dx ) \right),    
\end{align}
where $\kappa (x) = - \kappa (x)$ and $t \geq 0$, with $\kappa(x) = \boldsymbol{1}_{ [-1,1] }$. For example, if $t = 1$, then we have that 
\begin{align}
\mathbb{E} \big[  e^{ i \mathsf{z} L(1) } \big] = \mathsf{exp} \left(  - \frac{1}{2} c  z^2 - i b z +  \int_{ \mathbb{R} } \big(  e^{ i z x } - 1 - i z \kappa(x) \big) \nu( dx ) \right),    
\end{align}
\end{remark}
For infinite order linear processes with all coefficients of the same sign the idea is to approximate them by a sequence of finite order linear processes for which the theoretical conditions of the main theorem applies to the case of finite order linear processes with a truncation parameter $m$. Then, in order to verify that the equivalence limit process result holds, we need to determine the exact values of the characteristic triple process $V$. Moreover, to complete the proof we need to check that the stochastic process $L_n(t) := \big( V_n(t), W_n(t)  \big)$ jointly weakly converges within the $M_1$ topology which implies that each coordinate separately satisfies these theoretical assumptions and weakly convergence arguments. Then, proving that  $L_n$ converge in distribution to some asymptotic process, say $L = \big( V, W \big)$, in the standard $M_1$ topology, then using the fact that linear combinations of the coordinates of the joint functional are continuous and satisfy all related assumptions in the same topological space, then we can conclude that indeed $L_n$ weakly converge to its asymptotic counterpart $L$, endowed with the $M_1$ topology.

\subsection{Joint Functional Convergence of Partial Sum Processes with $\alpha \in [1,2)$}

In this section we concentrate on establishing the joint functional weak convergence result for the case where the regularly varying index $\alpha \in [1,2)$. In other words, we investigate the joint weak functional limiting behaviour of the partial sum and partial sum squares processes given by the stochastic process $L_n (.) = \left( L_{1n}(.), L_{2n}(.) \right)$ in the space $D\left( [0,1], \mathbb{R}^2 \right)$, where  
\begin{align}
L_{1n}(t) = \frac{ S_{ \floor{nt} } }{ a_n } - \floor{nt} b_{1n}, \ \ \ L_{2n}(t) =  \frac{ V^2_{ \floor{nt} } }{ a_n } - \floor{nt} b_{2n}, \ \ \ t \in [0,1],    
\end{align}
where $S_{ \floor{nt}} = \sum_{k=1}^{\floor{nt}} X_k$ and $V^2_{ \floor{nt}} = \sum_{k=1}^{\floor{nt}}  X^2_k $. Since $\left\{ a_n  \right\}$ is a sequence of positive real numbers such that $n \mathbb{P} \left( \left| X_1 \right| > a_n \right) \to 1$, as $n \to \infty$ and $L_n$ correspond to the centered stochastic process $L_n = \big( L_{1n}, L_{2n} \big)$ by the expected value of the truncated functional of each coordinate.
 
Therefore, we obtain the following expressions
\begin{align*}
L_{1n}(t) &=  \sum_{k=1}^{\floor{nt}} \frac{X_k}{a_n} - \floor{nt} \mathbb{E} \left[ \frac{X_1}{ a_n } \mathbf{1}_{ \left\{ u < \frac{ | X_1 | }{ a_n } \leq 1 \right\} } \right], \ \ \text{for} \ \alpha \in [1,2)
\\
L_{2n}(t) &= \sum_{k=1}^{\floor{nt}} \frac{X^2_k}{a^2_n} - \floor{nt} \mathbb{E} \left[ \frac{X^2_1}{ a^2_n } \mathbf{1}_{ \left\{ u < \frac{ X_1^2 }{ a^2_n } \leq 1 \right\} } \right], \ \ \text{for} \ \alpha \in [1,2) 
\end{align*}

\newpage

The constants $b_{1n}$ and $b_{2n}$ are defined with respect to the the regularly varying index as below
\begin{align}
\label{bn1}
b_{1n}
&:=
\begin{cases}
0, & \alpha \in (0,1), 
\\
\mathbb{E} \left[ \frac{X_1}{ a_n } \mathbf{1}_{ \left\{  \frac{ | X_1 | }{ a_n } \leq 1 \right\} } \right], & \alpha \in [1,2).
\end{cases}
\\
\label{bn2}
b_{2n}
&:=
\begin{cases}
0, & \alpha \in (0,1), 
\\
\mathbb{E} \left[ \frac{X^2_1}{ a^2_n } \mathbf{1}_{ \left\{  \frac{ X^2_1 }{ a^2_n } \leq 1  \right\} } \right], & \alpha \in [1,2).
\end{cases}
\end{align}
\begin{remark}
Based on the functional convergence result given above we can show that the weak limit of self-normalized partial-sum processes is indeed a stable L\'evy process where $\left\{ b_n \right\}$ is a sequence (see, \eqref{bn1} and \eqref{bn2}) such that it induces centered random variables for the stochastic processes under study. Specifically, for the self-normalized partial sum process, we need to prove that $S_{ \floor{n.} } / V_n \overset{ d }{ \to } L_1 (\,\cdot\,) \big/ \sqrt{L_2(1)}$, is an $\alpha-$stable L\'evy process 
\begin{align}
\frac{ S_{ \floor{n.} } }{ V_n } \overset{ d }{ \to } \frac{ L_1 (\,\cdot\,) }{ \sqrt{L_2(1)} } := \mathcal{L}_{\alpha} (\,\cdot\,)     
\end{align}
where the limit $\mathcal{L}_{\alpha} (\,\cdot\,)$ represents an $\alpha-$stable L\'evy process (see, Theorem \ref{stationary} and Theorem  \ref{t:functSN2}).
\end{remark}

By imposing the vanishing small values condition then convergence of the normalized partial sums can be established by showing that the limiting distribution  is infinitely divisible with a L\'evy triplet corresponding to an $\alpha-$stable distribution ( \cite{denker1989stable}, \cite{bartkiewicz2011stable}). The functional convergence of $L_n$ in the special case when $( X_n )_{n \in \mathbb{N} }$ is a linear process from a regularly varying distribution with index $\alpha \in (0,2)$ was studied in the literature previously. Therefore, the main result of our article shows that for a strictly stationary, regularly varying sequence of dependent random variables $(X_n)_{n \in \mathbb{N} }$ with index $\alpha \in (0,2)$, for which clusters of high-threshold excesses can be broken down into asymptotically independent blocks, the stochastic process $L_n$ converge in the space $D\left( [0,1], \mathbb{R}^2 \right)$ endowed with the Skorokhod weak $M_1$ topology under the condition that all extremes within each cluster of big values have the same sign. This topology is weaker than the more commonly used  Skorokhod $J_1$ topology, the latter being appropriate when there is no clustering of extremes (i.i.d case). 

Following similar argumentation as in the previous section, we will transfer the weak convergence of $N_n$ in expression \eqref{e:BaTa} to weak convergence of $L_n$. When $\alpha \in [1,2)$ we need the following conditions to deal with small jumps. Let $k = \floor{ \frac{n}{r} }$.
\begin{condition}
\label{cond2}
For all $\delta > 0$ it holds that, 
\begin{equation}
\lim_{u \to 0 } \limsup_{n \to \infty}
\mathbb{P} \left( \max_{ 0 \leq k \leq n } \left| \sum_{i=1}^k \left( \frac{X_i}{a_n} \mathbf{1}_{ \left\{ \frac{ |X_i| }{a_n} \leq u \right\} } - \mathbf{E} \left[ \frac{X_i}{a_n} \mathbf{1}_{ \left\{ \frac{|X_i| }{a_n} \leq u \right\} }   \right]  \right)  \right| > \delta  \right) = 0.
\end{equation}
\end{condition}

\newpage

\begin{condition}
\label{cond3}
For all $\delta > 0$ it holds that, 
\begin{equation}
\lim_{u \to 0 } \limsup_{n \to \infty}
\mathbb{P} \left( \max_{ 0 \leq k \leq n } \left| \sum_{i=1}^k \left( \frac{X_i^2}{a_n^2} \mathbf{1}_{ \left\{ \frac{ X_i^2 }{a^2_n} \leq u \right\} } - \mathbf{E} \left[ \frac{X_i^2}{a_n^2} \mathbf{1}_{ \left\{ \frac{X_i^2 }{a^2_n} \leq u \right\} }   \right]  \right)  \right| > \delta  \right) = 0.
\end{equation}
\end{condition}

\medskip

\begin{remark}
An equivalent of Condition \ref{cond1} but for the sequence of squared random variables holds, which we omit to present for space brevity. Moreover, Condition \ref{cond1} provides a statement for dependent sequences for certain type of mixing processes and it corresponds to the vanishing small values property. The index $\alpha$ is considered as a measure of the regular variation property rather than the degree of dependence of observations generated from the L\'{e}vy process, thus the above conditions could be challenging to prove for different mixing regimes. Condition \ref{cond2} and \ref{cond3} correspond to small jumps condition both for the $\left( X_n \right)_{n \in \mathbb{R}}$ as well as $\left( X^2_n \right)_{n \in \mathbb{R}}$ stationary, jointly regular varying sequences. A similar condition is given by Assumption 4.4 in \cite{basrak2018invariance}, which allows to prove that the finite dimensional marginal distributions of $V_n$ converge to those of an $\alpha-$stable L\'evy process.      
\end{remark}

\begin{theorem}
\label{stationary}
Let $(X_{n})_{n \in \mathbb{N} }$ be a stationary sequence of random variables, jointly regularly varying with index $\alpha \in [1,2)$, and of which the tail process $(Y_{i})_{i \in \mathbb{Z}}$ almost surely has no two values of the opposite sign. Suppose that Conditions~\ref{cond1} and~\ref{cond2} hold. Then, the stochastic process
\begin{align}
L_n(t) := \big( L_{1n}(t),  L_{2n}(t) \big) \equiv  \left(  \sum_{k=1}^{\floor{nt}} \frac{ X_k }{ a_n } - \floor{nt} b_{1n},  \sum_{k=1}^{\floor{nt}} \frac{ X^2_k }{ a^2_n }  - \floor{nt} b_{2n} \right), \ \ \ t \in [0,1],    
\end{align}
with $b_{1n}$ and $b_{2n}$ given by expressions \eqref{bn1} and \eqref{bn2} respectively. Then, it holds that
\begin{align}
 L_{n} \dto L \ \text{as} \ n \to \infty   
\end{align}
in $D([0,1], \mathbb{R}^{2})$ endowed with the weak $M_{1}$ topology, where $L = \big( L_{1}, L_{2} \big)$, $L_{1}$ is an $\alpha$--stable L\'{e}vy process with characteristic triple $(0, \nu_{1}, \gamma_{1})$ and $L_{2}$ is an $\alpha/2$--stable L\'{e}vy process with characteristic triple $(0, \nu_{2}, \gamma_{2})$, where
\begin{align}
\nu_{1}(\rmd x) &= \theta \alpha \big( c_{+} 1_{(0,\infty)}(x) + c_{-} \boldsymbol{1}_{(-\infty, 0)}(x) \big) |x|^{-\alpha -1}\,\rmd x,    
\\
\nu_{2}(\rmd x) &=  \frac{\theta \alpha}{2} \mathbf{E}\bigg( \sum_{j}\eta_{j}^{2}\bigg)^{\alpha/2} x^{-\alpha/2 -1} \boldsymbol{1}_{(0,\infty)}(x) \,\rmd x,
\\
c_{+} &= \mathbf{E} \bigg[ \Big( \sum_{j}\eta_{j} \Big)^{\alpha} \boldsymbol{1}_{ \{ \sum_{j}\eta_{j} > 0 \}} \bigg], \quad
c_{-} = \mathbf{E} \bigg[ \Big( - \sum_{j}\eta_{j} \Big)^{\alpha} \boldsymbol{1}_{ \{ \sum_{j}\eta_{j} < 0 \}} \bigg]    
\\
\gamma_{1} &=  \frac{\alpha}{\alpha - 1 } \bigg( p - q - \theta \big( c_{+}-c_{-} \big) \bigg). 
\end{align}
\end{theorem}

\newpage

\begin{remark}
Note that it can be proved that the two components of the limiting process $L = ( L_1, L_2)$ can be expressed as functionals of the limiting point process $N = \sum_i \sum_j \delta_{ ( T_i, P_i \eta_{ij} )}$. In particular, we define the limit process $L_1(\,\cdot\,)$ of the functional $L_{1n}(\,\cdot\,)$ according to the value of the regular varying index $\alpha$ such that   
\begin{align}
\label{L1expression}
L_1(\,\cdot\,) = 
\begin{cases}
\displaystyle \sum_{ T_i \leq . } \sum_j P_i \eta_{ij}, & \alpha \in (0,1),
\\
\displaystyle \underset{ u \to 0 }{ \mathsf{lim} } \left( \sum_{ T_i \leq . } \sum_j P_i \eta_{ij} 1_{ \left\{ P_i \left| \eta_{ij} \right| > u  \right\} } - (\,\cdot\,) \int_{  u < \left| x \right| \leq 1 } x \mu (dx) \right), & \alpha \in [1,2),
\end{cases}
\end{align}
where $L_1(\,\cdot\,)$ is an $\alpha-$stable L\'{e}vy process with characteristic triple $( 0, \nu_1, \gamma_1 )$, where the limit in the latter case holds almost surely uniformly on $[0,1]$ (along some sub-sequence).  

Furthermore, we define the limit process $L_2(\,\cdot\,)$ of the functional $L_{2n}(\,\cdot\,)$ according to the value of the regular varying index $\alpha$ such that   
\begin{align}
L_2(\,\cdot\,) = 
\begin{cases}
\displaystyle \sum_{ T_i \leq . } \sum_j P^2_i \eta^2_{ij}, & \alpha \in (0,1),
\\
\displaystyle \underset{ u \to 0 }{ \mathsf{lim} } \left( \sum_{ T_i \leq . } \sum_j P^2_i \eta^2_{ij} 1_{ \left\{ P^2_i \left| \eta_{ij} \right| > u  \right\} } - (\,\cdot\,) \int_{  u < \left| x \right| \leq 1 } x^2 \mu (dx) \right), & \alpha \in [1,2),
\end{cases}
\end{align}
where $L_2(\,\cdot\,)$ is an $\alpha/2-$stable L\'{e}vy process with characteristic triple $( 0, \nu_2, \gamma_2 )$, where the limit in the latter case holds almost surely uniformly on $[0,1]$ (along some sub-sequence).  
\end{remark}

\begin{proof}
For the proof of Theorem \ref{stationary} see Appendix \ref{AppendixB} of the paper. 
\end{proof}

\begin{remark}
Notice that in the case when $\alpha \in [1,2)$, we assume that the contribution of the smaller increments of the partial sum process is closer to its expectation. The random variable $X_1^2$ has an infinite expectation, but when we truncate this random variable with an appropriate indicator then we obtain finite expectation such that the following expression is finite 
\begin{align}
\mathbb{E} \left( \frac{ X_1^2 }{ \alpha_n^2 } \boldsymbol{1}_{ \frac{| X_1 |}{ \alpha_n } \leq 1 } \right) \leq 1 < \infty.  
\end{align}
Moreover, for $\alpha \in [1,2)$ the functionals are suitable centered with the coefficients $b_{1n}$ and $b_{2n}$ respectively. Then, the corresponding convergence results for $\alpha \in [1,2)$ only hold for finite quantities. Intuitively, Theorem \ref{stationary} shows that the joint functional consisting of the partial sum process in the first coordinate and the partial sum squared process in the second coordinate satisfy a functional central limit theorem in the weak $M_1$ topology with the limit consisting of an $\alpha-$stable L\'evy process in the first coordinate and an $\alpha/2-$stable L\'evy process in the second coordinate. In order to prove this result one needs to obtain the analytical form of their characteristic triples.  
\end{remark}

\newpage

\subsection{Weak Convergence of Self-Normalized Partial Sum Processes}
\label{S:Selfnormalized}

In this Section we establish the weak convergence of Self-normalized Partial Sum Processes in the $M_1$ topology. We use the main results established in Section \ref{Section2} of the paper applied to the case of the self-normalized partial sum process to obtain a weak convergence result that holds in the $M_1$ topology under general conditions. Consider the function $\mathsf{g} \colon D([0,1], \mathbb{R}^{2}) \to \mathbb{R}^{2}$, defined by $\mathsf{g}(x)=x(1)$. It is continuous with respect to the weak $M_{1}$ topology on $D([0,1], \mathbb{R}^{2})$ (see Theorem 12.5.2 in Whitt~\cite{Whitt2002stochastic}). Therefore under the conditions given by Theorem~\ref{t:functconvergence}, the continuous mapping theorem yields $\mathsf{g} (L_{n}) \dto \mathsf{g} (L)$ as $n \to \infty$, which implies that 
\begin{align}
\bigg( \sum_{i = 1}^{n} \frac{X_{i}}{a_{n}},  \sum_{i = 1}^{n} \frac{X_{i}^{2}}{a_{n}^{2}} \bigg)
\dto  \bigg( \sum_{i} \sum_{j}P_{i}\eta_{ij}, \sum_{i} \sum_{j}P_{i}^{2}\eta_{ij}^{2} \bigg).    
\end{align}
In particular, it follows that 
\begin{align}
\frac{V_{n}}{a_{n}^{2}} = \sum_{i = 1}^{n} \frac{X_{i}^{2}}{a_{n}^{2}}  \dto \sum_{i} \sum_{j}P_{i}^{2}\eta_{ij}^{2},    
\end{align}
where the limiting random variable has a stable distribution.

\begin{theorem}
\label{t:functSN1}
Let $(X_{n})_{ n \in \mathbb{Z} }$ be a stationary sequence of random variables, jointly regularly varying with index $\alpha\in(0,1)$, and of which the tail process $(Y_{i})_{i \in \mathbb{Z}}$ almost surely has no two values of the opposite sign. Suppose that Conditions~\ref{cond1} and~\ref{c:mixcond2} hold. Then
\begin{align}
\frac{S_{\lfloor n\,\cdot \rfloor}}{V_{n}} \dto \frac{L_{1}(\,\cdot\,)}{\sqrt{L_{2}(1)}} \qquad \textrm{as} \ n \to \infty,
\end{align}
in $D([0,1], \mathbb{R})$ endowed with the $M_{1}$ topology, where $L = (L_{1}, L_{2})$ as given in Theorem~\ref{t:functconvergence}.
\end{theorem}

\begin{proof}
The proof of Theorem \ref{t:functSN1} is shown in Appendix \ref{AppendixB} of the paper. 
\end{proof}

\begin{theorem}
\label{t:functSN2}
Let $(X_{n})_{ n \in \mathbb{Z} }$ be a stationary sequence of random variables, jointly regularly varying with index $\alpha\in[1,2)$, and of which the tail process $(Y_{i})_{i \in \mathbb{Z}}$ almost surely has no two values of the opposite sign. Suppose that Conditions~\ref{cond1} and~\ref{c:mixcond2} hold. Then
\begin{align}
\frac{S_{\lfloor n\,\cdot \rfloor}}{V_{n}} \dto \frac{L_{1}(\,\cdot\,)}{\sqrt{L_{2}(1)}} \qquad \textrm{as} \ n \to \infty,
\end{align}
in $D([0,1], \mathbb{R})$ endowed with the $M_{1}$ topology, where $L = (L_{1}, L_{2})$ as given in Theorem~\ref{t:functconvergence}.
\end{theorem}
\begin{proof}
The proof of Theorem \ref{t:functSN2} is shown in Appendix \ref{AppendixB} of the paper. 
\end{proof}
Both limit processes $L_1$ and $L_2$ denote independent strictly $\alpha-$stable and $\alpha / 2-$stable random variables respectively. Thus, our aim is to prove that the limit $L_{1}(\,\cdot\,) \big/ \sqrt{L_{2}(1)}$, that joint weak convergence induces in Theorem \ref{t:functSN1} and \ref{t:functSN2}, gives an $\alpha-$stable random variable. 

\newpage

\section{Main Applications}
\label{Section4}

In this section, we establish the weak convergence of joint functionals to the space $D[(0,1)]$ equipped with the $M_1$ topology for various processes. Our aim is to provide applications in which the finite-dimensional distributions of self-normalized partial-sum processes converge to those of a strictly $\alpha-$stable process and the partial-sum process is tight. We focus on demonstrating that the weak convergence of self-normalized processes of joint functionals of random variables induced by time series models hold to the $M_1$ topology based on our theoretical results. Thus, weak convergence arguments are essential for determining the limit behaviour of estimators and related test statistics that hold for sequences with a regular variation index $\alpha \in (0,2)$. In other words, we focus on establishing the weak convergence of the joint functional $L_n(t)$:  
\begin{align}
L_{n}(t) =  \left( \sum_{k=1}^{[nt]} \frac{X_{k}}{a_{n}}, \sum_{k=1}^{ \lfloor nt \rfloor}\frac{X_{k}^{2}}{a_{n}^{2}} \right),  \qquad t \in [0,1],  
\end{align}
for $\alpha \in (0,1)$ or $\alpha \in [1,2)$, where the stochastic process $\left( X_k \right)_{k \in \mathbb{N}}$ corresponds to the  stationary, regularly varying sequence induced from a regression model. The main idea here is to show that these results are consistent with the results provided by BKS for processes from a regularly varying distribution with index $\alpha \in (0,2)$ when one considers the weak convergence of the self-normalized partial sum functionals with appropriate centering and normalizing sequences.

We focus on time series processes for which the weak convergence of self-normalized partial sum processes do not converge in the standard $J_1$ topology. The particular weak convergence result is strong in the sense that it holds regardless whether clustering of only values of the same sign occurs in the limit process. In other words, the $J_1$ topology is appropriate when large values of $X_n$ do not cluster, while in the contrary existence of clusters of large values implies that $J_1$ convergence fail to hold (see, \cite{basrak2010functional}, \cite{krizmanic2016functional}). On the other hand, when considering the joint weak convergence of functionals within the $M_1$ topology related asymptotic theory results can be established. Therefore, we consider the weak convergence of such joint functionals to the $M_1$ topology, which is more appropriate in the presence of clusters of extreme values. Although, we have not mention any related mixing conditions, we mainly focus in the case of weak dependence and joint regular variation for the sequence $\left( X_k \right)_{k \in \mathbb{N}}$.  

\begin{example}[Linear process]

Let $\left( Z_i \right)_{ i \in \mathbb{Z} }$ be an \textit{i.i.d} regularly varying sequence with index $\alpha \in (0,2)$. Consider the following linear process for the \textit{i.i.d} sequence $\left( X_i \right)_{ i \in \mathbb{Z} }$  such that
\begin{align}
X_i = \sum_{j=0}^{ \infty } \varphi_j Z_{i-j}, \ \ \ \ k \in \mathbb{Z}, 
\end{align}
where the sequence of real numbers $( \varphi_j  )$ is such that 
\begin{align}
\sum_{ j = 0 }^{ \infty } \left| \varphi_j  \right|^{ \beta } < \infty \ \ \ \ \text{for some} \ \ 0 < \beta < \alpha,  \ \beta \leq 1,  \ \ \ \text{such that} \ \ \ \sum_{ j = 0 }^{ \infty } \varphi_j \neq 0.  
\end{align}

\newpage

Moreover, it holds that 
\begin{align}
\underset{ x \to \infty }{ \mathsf{lim} }  \frac{ \mathbb{P} \left( | X_0 | > x \right) }{ \mathbb{P} \left( | Z_0 | > x \right)  } = \sum_{j=0}^{\infty} \left| \varphi_j \right|^{\alpha}, \ \ \alpha \in (0,2). 
\end{align}
Then, the tail process $\left( Y_i \right)_{ k \in \mathbb{N} }$ of the linear process $\left( X_i \right)$ is of the following form:
\begin{align}
Y_i = \frac{ \varphi_{n+ M } }{ \left| \varphi_{M} \right| } | Y_0 | \Theta_i^{ Z_i  }, \ \ \ n \in \mathbb{Z},    
\end{align}
where $J$ is an $\left\{ -1, 1 \right\}$-valued random variable such that $\mathbb{P} \left( \Theta_i^{ Z_i } = 1 \right) = p$ and $\mathbb{P} \left( \Theta_i^{ Z_i } =  0 \right) = q$ for all $i \in \left\{ 1,.., n \right\}$ and $M$ is an integer valued random variable with probability distribution \begin{align}
\mathbb{P} \big( M = m \big) = \frac{ | \varphi_j |^{\alpha} }{ \sum_{i=0}^{\infty} | \varphi_j |^{\alpha}  }, \ m \geq 0.     
\end{align}
All the above assumptions ensure that the tail process has no two values of the opposite sign, which ensures that the theoretical conditions of our framework are satisfied. Thus, if we consider the corresponding finite order linear process, we can apply the main theory since in this case all conditions of our theorems are satisfied. 

Hence, for the truncated linear process: $X_i = \sum_{j=0}^m \varphi_j Z_{i-j}$ it holds that,  
\begin{align}
L_1(\,\cdot\,) = 
\begin{cases}
\displaystyle \frac{\sum_{j=0}^m \varphi_j }{ \kappa }  \sum_{ T_i \leq . } P_i \Theta_i^{ Z_i  }, & \alpha \in (0,1),
\\
\displaystyle \underset{ u \to 0 }{ \mathsf{lim} } \left( \frac{1}{\kappa}  \sum_{ T_i \leq . } P_i \Theta_i^{ Z_i  } \sum_{j=0}^m \varphi_j \mathbf{1}_{ \left\{ P_i \left| \varphi_j \right| / \kappa \right\} } - (\,\cdot\,) \int_{  u < \left| x \right| \leq 1 } x \mu (dx) \right), & \alpha \in [1,2),
\end{cases}
\end{align}
where $L_{1n}(t) := \frac{ 1 }{ a_n } \sum_{i=1}^{ \floor{nt} } X_i$ and $X_i = \sum_{j=0}^m \varphi_j Z_{i-j}$.

Therefore, we define with 
\begin{align}
L_2(\,\cdot\,) = 
\begin{cases}
\displaystyle \frac{\sum_{j=0}^m \varphi_j }{ \kappa }  \sum_{ T_i \leq . } P^2_i \left( \Theta_i^{ Z_i  } \right)^2, & \alpha \in (0,1),
\\
\displaystyle \underset{ u \to 0 }{ \mathsf{lim} } \left( \frac{1}{\kappa}  \sum_{ T_i \leq . } P^2_i \left( \Theta_i^{ Z_i  } \right)^2 \sum_{j=0}^m \varphi_j \mathbf{1}_{ \left\{ P^2_i \left| \varphi_j \right| / \kappa \right\} } - (\,\cdot\,) \int_{  u < \left| x \right| \leq 1 } x^2 \mu (dx) \right), & \alpha \in [1,2),
\end{cases}
\end{align}
where $L_{2n}(t) := \frac{ 1 }{ a^2_n } \sum_{i=1}^{ \floor{nt} } X^2_i$ and $X^2_i = \left( \sum_{j=0}^m \varphi_j Z_{i-j} \right)^2$.
\end{example}

Then, based on Theorem \ref{t:functSN1} and Theorem \ref{t:functSN2}, it holds that
\begin{align*}
\frac{ L_{ 1n}(t)  }{L_{2n}(1)}  
&\overset{\mathsf{def}}{= } 
\frac{ \displaystyle \frac{ 1 }{ a_n } \sum_{i=1}^{ \floor{nt} } X_i }{  \displaystyle  \frac{ 1 }{ a^2_n } \sum_{i=1}^{ n } X^2_i  }
= 
\frac{ \displaystyle \frac{ 1 }{ a_n } \sum_{i=1}^{ \floor{nt} } \sum_{j=0}^m \varphi_j Z_{i-j} }{  \displaystyle  \frac{ 1 }{ a^2_n } \sum_{i=1}^{ n } \left( \sum_{j=0}^m \varphi_j Z_{i-j} \right)^2  }
\end{align*}

\newpage

\begin{align*}
&\dto
\begin{cases}
\frac{  \displaystyle \frac{\sum_{j=0}^m \varphi_j }{ \kappa }  \sum_{ T_i \leq . } P_i \Theta_i^{ Z_i  }  }{  \displaystyle  \frac{\sum_{j=0}^m \varphi_j }{ \kappa }  \sum_{ T_i \leq . } P^2_i \left( \Theta_i^{ Z_i  } \right)^2  },  & \ \text{if} \ \alpha \in (0,1)     
\\
\\
\frac{ \displaystyle \underset{ u \to 0 }{ \mathsf{lim} } \left( \frac{1}{\kappa}  \sum_{ T_i \leq . } P_i \Theta_i^{ Z_i  } \sum_{j=0}^m \varphi_j \mathbf{1}_{ \left\{ P_i \left| \varphi_j \right| / \kappa \right\} } - t \int_{  u < \left| x \right| \leq 1 } x \mu (dx) \right) }{ \left\{ \displaystyle \underset{ u \to 0 }{ \mathsf{lim} } \left( \frac{1}{\kappa}  \sum_{ T_i \leq . } P^2_i \left( \Theta_i^{ Z_i  } \right)^2 \sum_{j=0}^m \varphi_j \mathbf{1}_{ \left\{ P^2_i \left| \varphi_j \right| / \kappa \right\} } - n \int_{  u < \left| x \right| \leq 1 } x^2 \mu (dx) \right)  \right\}^{1/2} },  & \ \text{if} \ \alpha \in [1,2)  \end{cases}
\\
&\equiv
\frac{L_{1}(\,\cdot\,)}{\sqrt{L_{2}(1)}}, \qquad \textrm{as} \ n \to \infty.
\end{align*}
\begin{remark}
Our first application considers a linear process which is a standard application in the literature (see also \cite{basrak2010functional}). Although, a linear process is considered a special case in which the $\beta-$mixing condition doesn't hold, our fundamental framework which corresponds to more general dependence structures encompasses the case of linear processes. Specifically, for this example, we focus on implementing our main results in order to establish the weak convergence of self-normalized partial sum processes of joint functionals equipped with the $M_1$ topology. On the other hand, in time series regression settings asymptotic independence conditions such as mixing conditions are usually proposed to replace independence, among which $\beta-$mixing is an important dependent structure and has been connected with a large class of time series models including ARMA models and GARCH models. 
\end{remark}

\begin{example}[Garch(1,1) process]

Consider the standard Garch model (see, \cite{bollerslev1986generalized}) 
\begin{align}
X_k = \sigma_k Z_k, \ \ \ k \in \mathbb{Z}.   
\end{align}
where $\left( Z_k \right)_{ k \in \mathbb{N} }$ is a sequence of \textit{i.i.d} random variables with $\mathbb{E}(Z_1) = 0$ and $\mathbb{E}(Z^2_1) = 1$ such that 
\begin{align}
\sigma_k^2 = \omega + \left( a_1 Z_{k-1}^2 + b_1 \right) \sigma_{k-1}^2, \ \ \ k \in \mathbb{Z}. 
\end{align}
where $a_1$ and $b_1$ are non-negative parameters ensuring the non-negativity of the squared volatility process $(\sigma_k^2)$. Moreover, we assume that $Z_1$ is symmetric, has a positive Levesgue density on $\mathbb{R}$ and there exists $\alpha \in (0,2)$ such that 
\begin{align}
\mathbb{E} \big[ \left( a_1 Z_1^2 + b_1 \right)^{\alpha} \big] = 1 \ \ \ \text{and} \ \ \ \mathbb{E} \big[ \left( a_1 Z_1^2 + \beta_1 \right)^{\alpha} \mathsf{ln} \left( a_1 Z_1^2 + b_1 \right) \big] < \infty.    
\end{align}
In other words, based on the well-known property that under fairly weak conditions on the distributions of the noise $(Z_t)$, its finite-dimensional distributions are regularly varying (see, \cite{mikosch2002whittle}) and thus satisfies the main conditions we impose in our framework.

\newpage

Specifically, it is known that the processes $\left( \sigma_k^2 \right)$ and $\left( X_k^2 \right)$ are regularly varying with index $\alpha$ and strongly mixing geometric rate. In particular, 
\begin{align}
\mathbb{P} \left( X_1^2 > x \right) \sim \mathbb{E} \left| Z_1 \right|^{2 \alpha} \mathbb{P} \left( \sigma_1^2 > x \right) \sim \mathbb{E} \left| Z_1 \right|^{ 2 \alpha } c_1 x^{- \alpha}, \ \ \text{as} \ \ x \to \infty.   
\end{align}

\end{example}

\begin{remark}
Garch(1,1) processes are stochastic volatility models used in the statistics and econometrics literature for modelling conditional volatility. Garch processes have a main advantage to Arch-type processes in capturing clusters of volatility or clusters of extreme values. Moreover, heavy-tailed Garch processes are found to be robust to capture volatility clustering effects in the literature. Therefore, in order for the conditions we impose in our framework to be satisfied, for the remaining derivations of this example, we consider partial sum processes of squared Garch(1,1) processes.    
\end{remark}

\medskip

\begin{example}
Consider the squared GARCH(1,1) process can be embedded in the 2-dimensional stochastic recurrence equation (SRE) such that 
\begin{align}
\boldsymbol{X}_k = \boldsymbol{A}_k \boldsymbol{X}_{k-1} + \boldsymbol{E}_k    
\end{align}
where 
\begin{align}
\boldsymbol{X}_k = 
\begin{pmatrix}
X_k^2
\\
\sigma_k^2
\end{pmatrix}, \ \ \
\boldsymbol{A}_k = 
\begin{pmatrix}
a_1 Z_k^2 \ & \ b_1 Z_k^2
\\
a_1 \ & \ b_1
\end{pmatrix}, \ \ \
\boldsymbol{X}_k = 
\begin{pmatrix}
\omega Z_k^2
\\
\omega
\end{pmatrix}.
\end{align}
Since the process $\left( X_k \right)$ is nonnegative, its tail process cannot have two values of the opposite sign. Moreover, the conditional volatility sequence $( \sigma_k^2 )_{ k \in \mathbb{Z} }$ has all elements of its time series observations being positive. Consequently, the vector process $\boldsymbol{X}_k$ satisfies the no sign switching condition (see, \cite{basrak2010functional}). Thus, we consider the weak convergence of the partial sum process $V_n(.)$ for the cases when $\alpha \in (0,1)$ and when $\alpha \in [1,2)$. 

The partial sum process $ \boldsymbol{L}_{1n} :=  V_n(.)$ is defined as below
\begin{align}
\label{par.proc}
\boldsymbol{V}_n(t) = \sum_{k=1}^{ \floor{nt} } \frac{ \boldsymbol{X}_k^2}{\alpha_n} - \floor{nt} \mathbb{E} \left[ \frac{ \boldsymbol{X}_1^2 }{\alpha_n} \mathbf{1}_{ \left\{ \frac{ \boldsymbol{X}_1^2 }{\alpha_n} \leq 1 \right\}} \right], \ t \in [0,1]    
\end{align}

Note that the partial sum process $\boldsymbol{V}_n(t)$ given by expression \eqref{par.proc} converges in distribution to an $\alpha-$stable L\'evy process $V(.)$, with a characteristic triple $\left( 0, \nu, b \right)$, in $D[0,1]$ equipped with the $M_1$ topology since it satisfies the Assumptions of the main theory of the paper.  Moreover, since $\boldsymbol{X}_k = \big( X_k^2, \sigma^2_k \big)^{\prime}$ is a vector stochastic process, for $k \in \left\{ 1,..., n \right\}$, the partial sum process $\boldsymbol{V}_n(t)$ can be expressed in the following matrix form

\newpage

\begin{align}
\label{par.proc}
\boldsymbol{L}_{1n} := \boldsymbol{V}_n(t) 
= 
\begin{pmatrix}
\displaystyle  \frac{1}{a_n} \sum_{k=1}^{ \floor{nt} }  X_k^2   
\\
\\
\displaystyle \frac{1}{a_n} \sum_{k=1}^{ \floor{nt} }   \sigma_k^2   
\end{pmatrix}
- 
\floor{nt} 
\begin{pmatrix}
\displaystyle \frac{1}{a_n} \mathbb{E} \left[ X_1^2 \mathbf{1}_{ \left\{ \frac{ \boldsymbol{X}_1^2 }{\alpha_n} \leq 1 \right\}} \right]   
\\
\\
\displaystyle \frac{1}{a_n} \mathbb{E} \left[ \sigma_1^2 \mathbf{1}_{ \left\{ \frac{ \sigma_1^2 }{\alpha_n} \leq 1 \right\}} \right]    
\end{pmatrix}
, \ t \in [0,1]   
\end{align}
Similarly, we define the partial-sum process $\boldsymbol{W}_n(t)$ as below
\begin{align}
\label{par.proc}
\boldsymbol{L}_{2n}  := \boldsymbol{W}_n(t) 
= 
\begin{pmatrix}
\displaystyle \frac{1}{a^2_n}  \sum_{k=1}^{ \floor{nt} } X_k^4   
\\
\\
\displaystyle \frac{1}{a^2_n} \sum_{k=1}^{ \floor{nt} } \sigma_k^4   
\end{pmatrix}
- 
\floor{nt} 
\begin{pmatrix}
\displaystyle \frac{1}{a^2_n} \mathbb{E} \left[ X_1^4 \mathbf{1}_{ \left\{ \frac{ \boldsymbol{X}_1^4 }{\alpha_n} \leq 1 \right\}} \right]   
\\
\\
\displaystyle \frac{1}{a^2_n} \mathbb{E} \left[ \sigma_1^4 \mathbf{1}_{ \left\{ \frac{ \sigma_1^4 }{\alpha_n} \leq 1 \right\}} \right]    
\end{pmatrix}
, \ t \in [0,1]   
\end{align}
Furthermore, by applying the main results of the paper we can consider the joint functional convergence of the corresponding self-normalized partial sum process such that 
\begin{align}
\frac{ \boldsymbol{L}_{ 1 \lfloor n t \rfloor} }{   \boldsymbol{L}_{2n}(\, t \,)  } \dto \frac{ \boldsymbol{L}_{1}(\, t \,)}{\sqrt{\boldsymbol{L}_{2}(\, 1 \,)}} \qquad \textrm{as} \ n \to \infty.
\end{align}
\end{example}

\begin{remark} The assumption that the tail process has no two values of the opposite sign is crucial to obtain weak convergence of the partial sum process in the $M_1$ topology. One of the restrictions of the particular topology is that it can handle several instantaneous jumps only if they are in the same direction. On the other hand, there appear to be some ways of omitting the no sign switching condition. A possibility is to avoid the within-cluster fluctuations in the partial sum process, by smoothing out its trajectories or by considering the process $t \mapsto S_{r_n \floor{k_n t }}$, which then implies that convergence holds in the stronger $J_1$ topology. Our main results build on the framework of BKS to establish the weak convergence of self-normalized partial sum processes within the more restrictive $M_1$ topology (see,  \cite{krizmanic2016functional, krizmanic2018joint, krizmanic2020joint}).
\end{remark}
 
Therefore, the above illustrate theoretical examples show cases in which the weak convergence of the partial-sum process, using its corresponding characteristic function representation, holds in the $D[0,1]$ space equipped with the $M_1$ topology. Similarly, one can prove that the second coordinate of the functional weakly converges to the $M_1$ topology which allows us to consider the joint weak convergence of the self-normalized partial-sum processes under investigation. To do this, we need to determine the exact values of the characteristic triple that correspond to both coordinates of the joint functional process. Specifically, for cases in which the standard $J_1$ topology fails to apply, we are interested to establish joint weak convergence in the $M_1$ topology which although is more restrictive it allows to obtain asymptotic results for extreme models and cluster-based processes which are more likely to have underline regular varying processes that satisfy the aforementioned theoretical results derived in this paper.

\newpage

\subsection{Application to Systemic Risk Measures}

\medskip

Systemic risk measures are commonly used for measuring stability in complex financial systems. 
These time series used for the purpose of evaluating systemic risk measures exhibit features such as heavy-tails and asymptotic tail independence which means that extreme values are less likely to occur simultaneously. Moreover, while the notion of asymptotic independence is  well studied in the bivariate case, it is less developed in a multivariate setting. Therefore, our aim is to define a notion which captures the global joint concurrent tail behavior of random vectors portrayed by many popular multivariate dependence structures which are amenable to extensions to any general dimensions. To do this we employ the Marshall-Olkin notion of dependence which allows to capture the tail dependence structure in multivariate time series jointly. 

In this direction, we can show that a multivariate time series vector of regularly varying sequences with Gaussian copula, mutual asymptotic independence holds only under certain restrictions on the correlation matrix; on the other hand, pairwise asymptotic independence holds for any choice of correlation matrix which has all pairwise correlations
less than one. Modelling heavy-tail time series is done based on the framework of multivariate regular variation which permits to compute tail probabilities even in the case of mutual or even pairwise asymptotic independence. 

\begin{definition}
A random vector $\boldsymbol{Z} \in \mathbb{R}^d$ is multivariate regularly varying on the topological space $\mathbb{E}_d^{(i)}$ if there exists a regularly varying function $b_i$ and a non-empty Borel measure $\mu_i$ which is finite on Borel sets bounded away from $\left\{  \boldsymbol{z} \in \mathbb{R}_{+}^d     \right\}$ such that \begin{align}
\mathbb{P} \left(  \frac{ \boldsymbol{Z} }{ b_i(t)  }  \in B    \right) = \mu_i(B).
\end{align}

\end{definition}

\newpage

\section{Conclusion}
\label{Section5}

\medskip

According to \cite{basrak2010functional}, for a stationary, regularly varying sequence for which clusters of high-thredshold excesses can be broken down into asymptotically independent blocks, the properly centered partial sum process $\left( V_n (t) \right)_{ t \in [0,1] }$ converges to an $\alpha-$stable L\'{e}vy process in the space $D[0,1]$ endowed with Skorokhod's $M_1$ metric under the condition that all extremes within one such cluster have the same sign. Therefore, the main proofs and applications in this paper are extending the particular conjecture to the case of self-normalized partial sum processes. 

In general, the $M_1$ topology is considered to be weaker than the more commonly used Skorohod $J_1$ topology, the $M_1$ topology allows to consider the case in which weak convergence has to be established under the presence of clustering of extreme values. This particular effect although is examined from the perspective of stationary stochastic processes, interesting questions arise in the case of nonstationary time series processes under the presence of clustering. Therefore, further research includes extensions to nonstationary sequences which is a special case. In that case, one would need to first of all check that the corresponding conditions and definitions (e.g., tail index of the sequence etc), hold and then prove the general case of the theorem. Then and iff the particular theorem holds we can apply the equivalent representation of self-normalized partial sum processes to nonstationary sequences without worrying that the weak convergence argument will not hold. Further interesting extensions of our work include among other cases such as when one considers joint weak convergence of empirical processes of cluster functionals (see, \cite{engelke2022graphical} and \cite{drees2010limit}). 

\bigskip

\paragraph{Acknowledgements}

I am grateful and indebted to Associate Professor Danijel Krizmanic from the Department of Mathematics, University of Rijeka for providing guidance throughout the development of this research project. Moreover, he provided the main steps for the proof in the case where $\alpha \in (0,1)$ which helped me establish the case $\alpha \in [1,2)$.

\newpage

\begin{appendix}

\section{Proofs of Main Results in Section 2}
\label{AppendixA}

\subsection{Proof of Lemma \ref{Lemma1}}

\begin{proof}
Since a continuous function on a compact set is bounded and attains its minimum value it holds that 
\begin{align}
 C^{\uparrow}_{0}([0,1], \mathbb{R}) = \bigcup_{k=1}^{\infty}C^{\uparrow}_{k}([0,1], \mathbb{R}),    
\end{align}
where $C^{\uparrow}_{k}([0,1], \mathbb{R})$ is the subset of functions $x$ in $C^{\uparrow}_{0}([0,1], \mathbb{R})$ with $x(0) \geq 1/k$. 

Consider $(x_{n})_{n}$ to be a sequence in $C^{\uparrow}_{k}([0,1], \mathbb{R})$ which converges to $x$ with respect to the $J_{1}$ topology. Since a discontinuous element of $D([0,1], \mathbb{R})$ cannot be approximated in the $J_{1}$ topology by a sequence of continuous functions, the limit $x$ must be continuous and the convergence is uniform (see for example \cite{kulik2020heavy}, Proposition C.2.4). 

Thus, for $0\leq s < t \leq 1$, from $x_{n}(s) \leq x_{n}(t)$ letting $n \to \infty$ we obtain $x(s) \leq x(t)$, i.e.~$x$ is nondecreasing. Similarly we obtain $x(0) \geq 1/k$, which means that $x \in C^{\uparrow}_{k}([0,1], \mathbb{R})$. Hence $C^{\uparrow}_{k}([0,1], \mathbb{R})$ is a closed, and then also a measurable subset of $D([0,1], \mathbb{R})$ with the $J_{1}$ topology. 

Since the Borel $\sigma$--fields associated with the non-uniform Skorokhod topologies all coincide with the Kolmogorov $\sigma$--field generated by the projection maps, the set $C^{\uparrow}_{k}([0,1], \mathbb{R})$ is measurable with the $M_{1}$ topology for each $k$, and therefore the same holds for $C^{\uparrow}_{0}([0,1], \mathbb{R})$.
\end{proof}

\subsection{Proof of Lemma \ref{l:M1div}}

\begin{proof}
\

Consider an arbitrary pair $(x,y) \in  D([0,1], \mathbb{R}) \times C^{\uparrow}_{0}([0,1], \mathbb{R})$. 

Suppose that $d_{p}((x_{n},y_{n}), (x,y)) \to 0$ as $n \to \infty$. We need to show that $h(x_{n},y_{n}) \to h(x,y)$ in the $M_{1}$ topology, that is, $d_{M_{1}}(x_{n}/y_{n}, x/y) \to 0$. By the definition of the metric $d_{p}$, we have $d_{M_{1}}(x_{n},x) \to 0$ and $d_{M_{1}}(y_{n},y) \to 0$ as $n \to \infty$. Since for monotone functions the $M_{1}$ convergence is equivalent to pointwise convergence in a dense subset $A$ of $[0,1]$ including $0$ and $1$ (\cite{Whitt2002stochastic}, Corollary 12.5.1), the convergence $d_{M_{1}}(y_{n},y) \to 0$, which implies that
\begin{align}
 \frac{1}{y_{n}(t)} \to \frac{1}{y(t)}, \qquad t \in A.    
\end{align}
Furthermore, since the functions $1/y, 1/y_{n}$ ($n \in \mathbb{N}$) are also monotone, $d_{M_{1}}(1/y_{n}, 1/y) \to 0$ as $n \to \infty$. Noting that the function $1/y$ is continuous, we can apply Lemma~\ref{l:M1div} to conclude that $d_{M_{1}}(x_{n}/y_{n}, x/y) \to 0$ as $n \to \infty$.
\end{proof}

\newpage 

\section{Proofs of Main Results in Section 3}
\label{AppendixB}

\subsection{Proof of Lemma \ref{l:contfunct}}

\begin{proof}
Take an arbitrary $\eta \in \Lambda$ and suppose that $\eta_{n} \vto \eta$ in $\mathbf{M}_p([0,1] \times \mathbb{E})$. We need to show that
$\Phi^{(u)}(\eta_n) \to \Phi^{(u)}(\eta)$ in $D([0,1], \mathbb{R}^{2})$ according to the weak $M_1$ topology. By Theorem 12.5.2 in \cite{Whitt2002stochastic}, it suffices to prove that, as $n \to \infty$,
\begin{align}
d_{p} \left( \Phi^{(u)}(\eta_{n}), \Phi^{(u)}(\eta) \right) =
\max_{k=1, 2}d_{M_{1}} \left( \Phi^{(u)}_{k}(\eta_{n}), \Phi^{(u)}_{k}(\eta) \right) \to 0.    
\end{align}
Following, with small modifications, the proof of Lemma~3.2 in \cite{basrak2010functional} for each coordinate separately, we obtain
$d_{M_{1}} \left( \Phi^{(u)}_{1}(\eta_{n}), \Phi^{(u)}_{1}(\eta) \right) \to 0$ and $d_{M_{1}} \left( \Phi^{(u)}_{2}(\eta_{n}), \Phi^{(u)}_{2}(\eta) \right) \to 0$ as $n \to \infty$. Therefore we conclude that $\Phi^{(u)}$ is continuous at $\eta$.
\end{proof}

\subsection{Proof of Theorem \ref{t:functconvergence}}

The following proof is based on the manuscript of Danijel Krijmanic who kindly shared. Related references to the main steps as presented in related studies in the literature are also carefully explained. Thus, in a similar manner as in \cite{krizmanic2020joint} who establish the joint weak convergence of partial sum and maxima processes, the following steps are adjusted accordingly in order to accommodate the fact that we are interested in establishing the joint weak convergence of the stochastic process $X_t$ and its square $X_t^2$ rather the corresponding maxima of this process. Consequently, the second coordinate of expression \eqref{X1X2} corresponds to the squared term (e.g., see proof of Theorem 3.4 in \cite{krizmanic2020joint}).  Take an arbitrary $u>0$, and consider
\begin{align}
\label{X1X2}
\Phi^{(u)}(N_{n})(t) 
= 
\left( \sum_{i/n \leq t}\frac{X_{i}}{a_{n}} \boldsymbol{1}_{ \big\{ \frac{|X_{i}|}{a_{n}} > u \big\} }, \sum_{i/n \leq t}\frac{X_{i}^{2}}{a_{n}^{2}} \boldsymbol{1}_{ \big\{ \frac{|X^{2}_{i}|}{a^{2}_{n}} > u \big\} } \right), \qquad t \in [0,1].    
\end{align}

From Lemma~\ref{l:prob1} and Lemma~\ref{l:contfunct} we know that $\Phi^{(u)}$ is continuous on the set $\Lambda$ defined by \eqref{lampdas1}.  Recall that this set is basically an extension of the space $E$ as defined in Section 4.2 of \cite{basrak2018invariance}, in order to accommodate the more general setting we have in this study. This set almost surely contains the limiting point process $N$ such that  $N_{n} \dto N := \sum_{i}\sum_{j}\delta_{(T_{i}, P_{i}\eta_{ij})} \in [0,1] \times \mathbb{E} \ \ \ \text{as} \ n \to \infty$. 

Hence an application of the continuous mapping theorem yields that $\Phi^{(u)}(N_{n}) \dto \Phi^{(u)}(N)$ in $D([0,1], \mathbb{R}^{2})$ under the weak $M_{1}$ topology (see, similar argumentation in \cite{krizmanic2018joint, krizmanic2020joint}.  Therefore, we can establish the convergence of the functional $L_{n}^{(u)}(\,\cdot\,)$ as below
\begin{multline}
\label{e:mainconv}
  L_{n}^{(u)}(\,\cdot\,) 
      := 
      \left( \sum_{i = 1}^{\lfloor n \, \cdot \, \rfloor} \frac{X_{i}}{a_{n}}
    \boldsymbol{1}_{ \bigl\{ \frac{|X_{i}|}{a_{n}} > u \bigr\} },  \sum_{i = 1}^{\lfloor n \, \cdot \, \rfloor} \frac{X_{i}^{2}}{a_{n}^{2}}
    \boldsymbol{1}_{ \bigl\{ \frac{|X^{2}_{i}|}{a^{2}_{n}} > u \bigr\} } \right) 
    \\
    \dto L^{(u)}(\,\cdot\,) 
    :=  
    \left( \sum_{T_{i} \le \, \cdot} \sum_{j}P_{i}\eta_{ij} \boldsymbol{1}_{\{ P_{i}|\eta_{ij}| > u \}}, \sum_{T_{i} \le \, \cdot} \sum_{j}P_{i}^{2}\eta_{ij}^{2} \boldsymbol{1}_{\{ P_{i}|\eta_{ij}| > u \}} \right).
\end{multline}

\newpage

where $\left\{ a_n \right\}$ is a sequence of norming constants which ensures that convergence in distribution holds as $u \to 0$. 

Moreover, let $U_{i} = \sum_{j}\eta_{ij}$ and $\widetilde{U}_{i} = \sum_{j}\eta_{ij}^{2}$, where $i \in \left\{ 1,2,\ldots \right\}$. Then, since in Theorem 1 we assume that $\alpha < 1$, from the moment condition given in Corollary 1 (\ref{e:MW1}) (e.g., see \cite{davis1995point}), the following conditions apply for the two coordinates of the functional under consideration.  
\begin{align*}
\label{e:MW}
\mathrm{E}|U_{1}|^{\alpha} &\leq \mathrm{E} \Big( \sum_{j}|\eta_{j}| \Big)^{\alpha} < \infty,
\\
\mathrm{E}|\widetilde{U}_{1}|^{\alpha/2} = \mathrm{E} \Big( \sum_{j}|\eta_{j}|^{2} \Big)^{\alpha/2} &\leq \mathrm{E} \Big( \sum_{j}|\eta_{j}| \Big)^{\alpha} < \infty,
\end{align*}
where in the last relation we used also the triangle inequality $|\sum_{i}a_{i}|^{s} \leq \sum_{i}|a_{i}|^{s}$ with $s \in (0,1]$.

Using Proposition 5.2 and Proposition 5.3 given in \cite{resnick2007heavy}, it holds that $P_{i}|U_{i}|$, for $i \in \left\{ 1,2,\ldots \right\}$, are the points of a Poisson process with intensity measure $\theta \alpha \mathbf{E} |U_{1}|^{\alpha} x^{-\alpha-1}\,\rmd x$, where $x>0$. Since these points are summable it holds that 
\begin{align}
\sum_{i=1}^{\infty}P_{i}|U_{i}| \leq \sum_{i=1}^{\infty}\sum_{j}P_{i}|\eta_{ij}| < \infty    
\end{align}
almost surely (see the proof of Theorem 3.1 in \cite{davis1995point}). Thus for all $t\in [0,1]$
\begin{align}
\label{sum.condition}
\sum_{T_{i} \leq .}^t \sum_{j} P_{i}\eta_{ij} \boldsymbol{1}_{\{ P_{i}|\eta_{ij}| > u \}} \to \sum_{T_{i} \leq t} \sum_{j}P_{i}\eta_{ij}    
\end{align}
almost surely as $u \to 0$. In other words, expression \eqref{sum.condition} provides a summability condition that ensures weak convergence of a partial sum process to L\'evy processes (e.g., see \cite{krizmanic2020joint}). 

Furthermore, the first summand has a lower bound $\left\{ T_i \leq . \right\}$ and an upper bound $t$ which implies that we are summing up all points up to until time $T_i$. 

By the dominated convergence theorem
\begin{align}
\sup_{t \in [0,1]} \bigg| \sum_{T_{i} \leq t} \sum_{j} P_{i}\eta_{ij} \boldsymbol{1}_{\{ P_{i}|\eta_{ij}| > u \}} - \sum_{T_{i} \leq t} \sum_{j}P_{i}\eta_{ij} \bigg| \leq \sum_{i=1}^{\infty}\sum_{j}P_{i}|\eta_{ij}| \boldsymbol{1}_{ \{ P_{i}|\eta_{ij}| \leq u \} } \to 0  \end{align}
almost surely as $u \to 0$.

Since uniform convergence implies Skorokhod $M_{1}$ convergence, we get
\begin{equation}
\label{e:dm1}
d_{M_{1}} \left( \sum_{T_{i} \leq\,\cdot} \sum_{j} P_{i}\eta_{ij} \boldsymbol{1}_{\{ P_{i}|\eta_{ij}| > u \}}, \sum_{T_{i} \leq\,\cdot} \sum_{j}P_{i}\eta_{ij} \right) \to 0
\end{equation}
almost surely as $u \to 0$. 

\newpage

Similarly, $P_{i}^{2}\widetilde{U}_{i}$, $i \in \left\{ 1,2,\ldots \right\}$, are the points of a Poisson process with intensity measure $(\theta \alpha /2) \mathrm{E}(\widetilde{U}_{1})^{\alpha/2} x^{-\alpha/2-1}\,\rmd x$ for $x>0$, and it holds that
\begin{equation}
\label{e:dm12}
d_{M_{1}} \left( \sum_{T_{i} \leq\,\cdot} \sum_{j} P_{i}^{2}\eta_{ij}^{2} \boldsymbol{1}_{\{ P_{i}|\eta_{ij}| > u \}}, \sum_{T_{i} \leq\,\cdot} \sum_{j}P_{i}^{2}\eta_{ij}^{2} \right) \to 0
\end{equation}
almost surely as $u \to 0$.
Recalling the definition of the metric $d_{p}$ in (\ref{e:defdp}), from (\ref{e:dm1}) and (\ref{e:dm12}) we obtain that almost surely as $u \to 0$, 
\begin{equation*}
d_{p} \left( L^{(u)}, L \right) \to 0
\end{equation*}
where
\begin{align}
L(t) = \left( \sum_{T_{i} \le t} \sum_{j}P_{i}\eta_{ij}, \sum_{T_{i} \le t} \sum_{j}P_{i}^{2}\eta_{ij}^{2} \right).    
\end{align}
Since almost sure convergence implies weak convergence, we have, as $u \to 0$,
\begin{equation}
\label{e:mainconv2}
 L^{(u)} \dto L
\end{equation}
in $D([0,1], \mathbb{R}^{2})$ endowed with the weak $M_{1}$ topology.
By Proposition 5.2 and Proposition 5.3 in \cite{resnick2007heavy}, the process
\begin{align}
\sum_{i} \delta_{(T_{i}, \sum_{j}P_{i}\eta_{ij} })    
\end{align}
is a Poisson process with intensity measure $Leb \times \nu_{1}$, and
\begin{align}
\sum_{i} \delta_{(T_{i}, \sum_{j}P_{i}^{2}\eta_{ij}^{2}})    
\end{align}
is a Poisson process with intensity measure $Leb \times \nu_{2}$.

Furthermore, by the It\^{o} representation of the L\'{e}vy process (see \cite{resnick2007heavy}, pages 150--153) and Theorem 14.3 in \cite{ken1999levy},
\begin{align}
L_{1} (t) = \sum_{T_{i} \leq t} \sum_{j}P_{i}\eta_{ij}, \qquad t \in [0,1],    
\end{align}
is an $\alpha$--stable L\'{e}vy process with characteristic triple $(0, \nu_{1}, (c_{+}-c_{-}) \theta \alpha / (1-\alpha))$, and similarly
\begin{align}
L_{2} (t) = \sum_{T_{i} \leq t} \sum_{j}P_{i}^{2}\eta_{ij}^{2}, \qquad t \in [0,1],    
\end{align}
is an $\alpha/2$--stable L\'{e}vy process with characteristic triple $(0, \nu_{2}, \theta \alpha \mathrm{E}(\sum_{j}\eta_{j}^{2})^{\alpha/2}/(2-\alpha))$.

\newpage

Specifically, notice that the functional $L_1$ is an $\alpha-$stable L\'{e}vy process while the $L_2$ functional represents $\alpha/2$--stable L\'{e}vy process. Furthermore, within the same probability space it is well-known that the sum of two independent L\'{e}vy processes gives a stable L\'{e}vy process with intensity measure the sum of the intensity measures of the two processes. 

Therefore, if we show that
\begin{align}
\lim_{u \to 0}\limsup_{n \to \infty} \mathbb{P} \left( d_{p}(L_{n},L_{n}^{(u)}) > \epsilon \right)=0
\end{align}
for any $\epsilon >0$, from (\ref{e:mainconv}) and (\ref{e:mainconv2}) by a variant of Slutsky's theorem (see Theorem 3.5 in \cite{resnick2007heavy}) it will follow that $ L_{n} \dto L$ as $n \to \infty$, in $D([0,1], \mathbb{R}^{2})$ with the weak $M_{1}$ topology.

Since the  metric $d_{p}$ on $D([0,1], \mathbb{R}^{2})$ is bounded above by the uniform metric on  $D([0,1], \mathbb{R}^{2})$ (Theorem 12.10.3 in \cite{Whitt2002stochastic}), it suffices to show that
\begin{align}
\lim_{u \to 0} \limsup_{n \to \infty} \mathbb{P} \biggl( \sup_{t \in [0,1]} \|L_{n}(t) - L_{n}^{(u)}(t)\| > \epsilon \biggr)=0.    
\end{align}
Using stationarity and Markov's inequality we get the bound 
\begin{eqnarray}
\label{e:slutsky}
\nonumber 
\mathbb{P} \bigg(
     \sup_{t \in [0,1]} \|L_{n}(t) - L_{n}^{(u)}(t)\| >  \epsilon \bigg) & & 
     \\
   \nonumber & \hspace*{-20em} = & \ \hspace*{-10em} \mathbb{P} \bigg(
       \sup_{t \in [0,1]} \ \max \bigg\{  \bigg| \sum_{i=1}^{\lfloor nt \rfloor} \frac{X_{i}}{a_{n}}
       \boldsymbol{1}_{ \big\{ \frac{|X_{i}|}{a_{n}} \leq u \big\} }  \bigg|, \bigg| \sum_{i=1}^{\lfloor nt \rfloor} \frac{X_{i}^{2}}{a_{n}^{2}}
       \boldsymbol{1}_{ \big\{ \frac{|X^{2}_{i}|}{a^{2}_{n}} \leq u \big\} } \bigg| \bigg\}  > \epsilon
       \bigg)
       \\[0.4em]
   \nonumber  & \hspace*{-20em} \leq & \ \hspace*{-10em} \mathbb{P} \bigg(
       \sum_{i=1}^{n} \frac{|X_{i}|}{a_{n}}
       \boldsymbol{1}_{ \big\{ \frac{|X_{i}|}{a_{n}} \leq u \big\} }  > \epsilon
       \bigg) + \mathbb{P} \bigg(
       \sum_{i=1}^{n} \frac{X_{i}^{2}}{a_{n}^{2}}
       \boldsymbol{1}_{ \big\{ \frac{|X^{2}_{i}|}{a^{2}_{n}} \leq u \big\} }  > \epsilon
       \bigg)
       \\[0.4em]
      & \hspace*{-20em} \leq & \ \hspace*{-10em}
      \epsilon^{-1} n \mathbf{E} \bigg( \frac{|X_{1}|}{a_{n}}
       \boldsymbol{1}_{ \big\{ \frac{|X_{1}|}{a_{n}} \leq u \big\} } \bigg) +  \epsilon^{-1} n \mathbf{E} \bigg( \frac{X_{1}^{2}}{a_{n}^{2}}
       \boldsymbol{1}_{ \big\{ \frac{|X^{2}_{1}|}{a^{2}_{n}} \leq u \big\} } \bigg).
\end{eqnarray}
For the first term on the right-hand side of (\ref{e:slutsky}) we have
\begin{align}
n \mathbf{E} \bigg( \frac{|X_{1}|}{a_{n}}
       \boldsymbol{1}_{ \big\{ \frac{|X^{2}_{1}|}{a^{2}_{n}} \leq u \big\} } \bigg) =  u \cdot n \mathbb{P} (|X_{1}| > a_{n}) \cdot \frac{ \mathbb{P} (|X_{1}| > ua_{n})}{ \mathbb{P} (|X_{1}|>a_{n})} \cdot
          \frac{\mathbf{E}(|X_{1}| \boldsymbol{1}_{ \{ |X_{1}| \leq ua_{n} \} })}{ ua_{n} \mathbb{P} (|X_{1}| >ua_{n})},    
\end{align}
and for the second
\begin{align}
n \mathbf{E} \bigg( \frac{X_{1}^{2}}{a_{n}^{2}}
       \boldsymbol{1}_{ \big\{ \frac{|X_{1}|}{a_{n}} \leq u \big\} } \bigg) =  u^{2} \cdot n \mathbb{P} (|X_{1}| > a_{n}) \cdot \frac{\mathbb{P}(|X_{1}| > ua_{n})}{\mathbb{P}(|X_{1}|>a_{n})} \cdot
          \frac{\mathbf{E}(|X_{1}|^{2} \boldsymbol{1}_{ \{ |X_{1}| \leq ua_{n} \} })}{ u^{2}a_{n}^{2} \mathbb{P} (|X_{1}| >ua_{n})},    
\end{align}
Since $X_{1}$ is a regularly varying random variable with index $\alpha$, it follows immediately that
\begin{align}
 \frac{\mathbb{P}(|X_{1}| > ua_{n})}{\mathbb{P}(|X_{1}|>a_{n})} \to u^{-\alpha}    
\end{align}
as $n \to \infty$. By Karamata's theorem
\begin{align}
\lim_{n \to \infty} \frac{\mathbf{E}(|X_{1}| \boldsymbol{1}_{ \{ |X_{1}| \leq ua_{n} \} })}{ ua_{n} \mathbb{P} (|X_{1}| >ua_{n})} = \frac{\alpha}{1-\alpha} \qquad \textrm{and} \qquad \lim_{n \to \infty} \frac{\mathbf{E}(|X_{1}|^{2} \boldsymbol{1}_{ \{ |X_{1}| \leq ua_{n} \} })}{ u^{2}a_{n}^{2} \mathbb{P} (|X_{1}| >ua_{n})} = \frac{\alpha}{2-\alpha}.     
\end{align}
Therefore, taking into account relation (\ref{e:niz}), we get
\begin{align}
 n \mathbf{E} \bigg( \frac{|X_{1}|}{a_{n}}
       \boldsymbol{1}_{ \big\{ \frac{|X_{1}|}{a_{n}} \leq u \big\} } \bigg) \to u^{1-\alpha} \frac{\alpha}{1-\alpha} \qquad \textrm{and} \qquad n \mathbf{E} \bigg( \frac{X_{1}^{2}}{a_{n}^{2}}
       \boldsymbol{1}_{ \big\{ \frac{|X_{1}|}{a_{n}} \leq u \big\} } \bigg) \to u^{2-\alpha} \frac{\alpha}{2-\alpha}    
\end{align}
as $n \to \infty$. Therefore from (\ref{e:slutsky}) we obtain
\begin{align}
\limsup_{n \to \infty} \mathbb{P} \bigg( \sup_{t \in [0,1]} \|L_{n}(t) - L_{n}^{(u)}(t)\| >  \epsilon \bigg) \leq \epsilon^{-1} \alpha u^{1-\alpha} \Big( \frac{1}{1-\alpha} + \frac{u}{2-\alpha} \Big).    
\end{align} 
Letting $u \to 0$, since $1-\alpha >0$, we finally obtain
\begin{align}
\lim_{u \to 0} \limsup_{n \to \infty} \mathbb{P} \bigg( \sup_{t \in [0,1]} \|L_{n}(t) - L_{n}^{(u)}(t)\| >  \epsilon \bigg) = 0.    
\end{align} 


\newpage

\subsection{Proof of Theorem \ref{stationary}}

Notice that it holds that $b_{1n}^{(u)} \to b_1$ and $b_{2n}^{(u)}\to b_2$ as $u \to 0$
\begin{align}
b_{1n}^{(u)} 
&:=
\begin{cases}
0, & \alpha \in (0,1), 
\\
\mathbb{E} \left[ \frac{X_1}{ a_n } \mathbf{1}_{ \left\{ \textcolor{red}{u < } \frac{ | X_1 | }{ a_n } \leq 1 \right\} } \right], & \alpha \in [1,2).
\end{cases}
\\
b_{2n}^{(u)} 
&:=
\begin{cases}
0, & \alpha \in (0,1), 
\\
\mathbb{E} \left[ \frac{X^2_1}{ a^2_n } \mathbf{1}_{ \left\{ \textcolor{red}{u  < } \frac{ X_1^2 }{ a_n } \leq 1  \right\} } \right], & \alpha \in [1,2).
\end{cases}
\end{align}
\begin{align}
b_{1}
&:=
\begin{cases}
0, & \alpha \in (0,1), 
\\
\displaystyle  \int_{ \textcolor{red}{0< }|x| \leq 1} x \mu (dx), & \alpha \in [1,2).
\end{cases}
\\
b_{2}
&:=
\begin{cases}
0, & \alpha \in (0,1), 
\\
\displaystyle  \int_{ \textcolor{red}{0 < } x^2 \leq 1} x^2 \mu (dx), & \alpha \in [1,2).
\end{cases}
\end{align}
In other words, in the case where the regularly varying index $\alpha \in [1,2)$ similar to the approach of \cite{basrak2018invariance} (e.g., see Section 4.2), we consider the centering constant.

\begin{proof}
Therefore, similar to the case $\alpha \in (0,1)$ we aim to prove the joint weak convergence of the functional $L_n(t)$ of both its coordinates to the corresponding limit process. To do this, we take an arbitrary $u>0$, and consider
\begin{align}
\Phi^{(u)}(N_{n})(t) 
= 
\left( \sum_{i/n \leq t}\frac{X_{i}}{a_{n}} \mathbf{1}_{ \left\{  \frac{ | X_1 | }{ a_n } \leq u \right\} }, \sum_{i/n \leq t}\frac{X_{i}^{2}}{a_{n}^{2}} \mathbf{1}_{ \left\{  \frac{ |X_1| }{ a_n } \leq u  \right\} }  \right), \qquad t \in [0,1].    
\end{align}
From Lemma~\ref{l:prob1} and Lemma~\ref{l:contfunct} we know that $\Phi^{(u)}$ is continuous on the set $\Lambda$ and this set almost surely contains the limiting point process $N$ from (\ref{e:BaTa}). Hence an application of the continuous mapping theorem yields that $\Phi^{(u)}(N_{n}) \dto \Phi^{(u)}(N)$ in $D([0,1], \mathbb{R}^{2})$ under the weak $M_{1}$ topology. Therefore, we can establish the convergence of the functional $L_{n}^{(u)}(\,\cdot\,)$ as below
\begin{multline}
\label{e:mainconv}
L_{n}^{(u)}(\,\cdot\,) 
:= 
\left( \sum_{i = 1}^{\lfloor n \, \cdot \, \rfloor} \frac{X_{i}}{a_{n}}
       - \lfloor n \, \cdot \, \rfloor b_{1n}^{(u)},  \sum_{i = 1}^{\lfloor n \, \cdot \, \rfloor} \frac{X_{i}^{2}}{a_{n}^{2}} - \lfloor n \, \cdot \, \rfloor b_{2n}^{(u)}  \right) 
    \\
    \dto L^{(u)}(\,\cdot\,) 
    :=  
    \left( \underset{ u \to 0  }{ \mathsf{lim} } \left\{ \sum_{T_{i} \le \, \cdot} \sum_{j}P_{i}\eta_{ij} \boldsymbol{1}_{\{ P_{i}|\eta_{ij}| > u \}} - (\,\cdot\,)  b_{1} \right\},   \underset{ u \to 0  }{ \mathsf{lim} } \left\{ \sum_{T_{i} \le \, \cdot} \sum_{j}P_{i}^{2}\eta_{ij}^{2} \boldsymbol{1}_{\{ P^2_{i}|\eta^2_{ij}| > u \}} - (\,\cdot\,)  b_{2}  \right\} \right).
\end{multline}
where $\left\{ a_n \right\}$ is a sequence of norming constants which ensures that convergence in distribution holds as $u \to 0$, which proves that $L^{ (u) } \overset{ d }{ \to }  L$ in $\mathcal{D}  ( [0,1], \mathbb{R}^2 )$ endowed with the weak $M_1$ topology.

\medskip

\newpage

Notice that the above joint convergence results implies that for each coordinate of the joint functional it holds that 
\begin{align}
\sum_{i = 1}^{\lfloor n \, \cdot \, \rfloor} \frac{X_{i}}{a_{n}} - \lfloor n \, \cdot \, \rfloor b_{1n}^{(u)} &\to \underset{ u \to 0  }{ \mathsf{lim} } \left\{ \sum_{T_{i} \le \, \cdot} \sum_{j}P_{i}\eta_{ij} \boldsymbol{1}_{\{ P_{i}|\eta_{ij}| > u \}} - (\,\cdot\,)  b_{1} \right\}, \ \ \ \text{as} \ n \to \infty
\\
\sum_{i = 1}^{\lfloor n \, \cdot \, \rfloor} \frac{X_{i}^{2}}{a_{n}^{2}} - \lfloor n \, \cdot \, \rfloor b_{2n}^{(u)}  &\to  \underset{ u \to 0  }{ \mathsf{lim} } \left\{ \sum_{T_{i} \le \, \cdot} \sum_{j}P_{i}^{2}\eta_{ij}^{2} \boldsymbol{1}_{\{ P^2_{i}|\eta^2_{ij}| > u \}} - (\,\cdot\,)  b_{2}  \right\}, \ \ \ \text{as} \ n \to \infty
\end{align}
\color{black}
or in other words, 
\begin{align}
\sum_{i = 1}^{\lfloor n \, \cdot \, \rfloor} \frac{X_{i}}{a_{n}} - \lfloor n \, \cdot \, \rfloor \mathbb{E} \left[ \frac{X_1}{ a_n } \mathbf{1}_{ \left\{ \textcolor{red}{u < } \frac{ | X_1 | }{ a_n } \leq 1 \right\} } \right] 
\to
\underset{ u \to 0  }{ \mathsf{lim} } \left\{ \sum_{T_{i} \le \, \cdot} \sum_{j}P_{i}\eta_{ij} \boldsymbol{1}_{\{ P_{i}|\eta_{ij}| > u \}} - (\,\cdot\,)  \displaystyle  \int_{ \textcolor{red}{u < }|x| \leq 1} x \mu (dx) \right\}, \ \text{for} \ \ \alpha \in [1,2).   
\end{align}
and
\begin{align}
\sum_{i = 1}^{\lfloor n \, \cdot \, \rfloor} \frac{X^2_{i}}{a^2_{n}} - \lfloor n \, \cdot \, \rfloor \mathbb{E} \left[ \frac{X^2_1}{ a^2_n } \mathbf{1}_{ \left\{ \textcolor{red}{u < } \frac{ | X_1 | }{ a_n } \leq 1 \right\} } \right] 
\to
\underset{ u \to 0  }{ \mathsf{lim} } \left\{ \sum_{T_{i} \le \, \cdot} \sum_{j}P^2_{i}\eta^2_{ij} \boldsymbol{1}_{\{ P^2_{i}|\eta^2_{ij}| > u \}} - (\,\cdot\,)  \displaystyle  \int_{ \textcolor{red}{u < } x^2 \leq 1} x^2 \mu (dx) \right\}, \ \text{for} \ \ \alpha \in [1,2).   
\end{align}

\medskip

Next, as before we need to show that
\begin{align}
\lim_{u \to 0} \limsup_{n \to \infty} \Pr \biggl( \sup_{t \in [0,1]} \|L_{n}(t) - L_{n}^{(u)}(t)\| > \epsilon \biggr) = 0.   
\end{align}
Using stationarity and Markov's inequality we get the bound
\begin{small}
\begin{align}
\label{e:slutsky}
\nonumber 
&\mathbb{P} \bigg(
\sup_{t \in [0,1]} \|L_{n}(t) - L_{n}^{(u)}(t)\| >  \epsilon \bigg) 
\\[0.3em]
\nonumber 
&=  
\mathbb{P} \bigg(
       \sup_{t \in [0,1]} \ \max \bigg\{  \bigg| \sum_{i=1}^{\lfloor nt \rfloor} \frac{X_{i}}{a_{n}}
       \mathbf{1}_{ \big\{ \frac{|X_{i}|}{a_{n}} \leq u \big\} }  - \floor{nt} \mathbb{E} \left(  \frac{X_{1}}{a_{n}} \mathbf{1}_{ \big\{ \frac{|X_{1}|}{a_{n}} \leq u \big\} }  \right)   \bigg|, \bigg| \sum_{i=1}^{\lfloor nt \rfloor} \frac{X_{i}^2 }{a^2_{n}}
       \mathbf{1}_{ \big\{ \frac{|X_{i}|}{a_{n}} \leq u \big\} }  - \floor{nt} \mathbb{E} \left(  \frac{X^2_{1}}{a^2_{n}} \mathbf{1}_{ \big\{ \frac{|X_{1}|}{a_{n}} \leq u \big\} }  \right)   \bigg|  \bigg\}  > \epsilon
       \bigg)
       \\[0.4em]
   \nonumber 
&=
\mathbb{P} \bigg(
       \max_{ 0 \leq k \leq n } \ \bigg| \sum_{i=1}^{ k } \bigg\{  \frac{X_{i}}{a_{n}}
       \mathbf{1}_{ \big\{ \frac{|X_{i}|}{a_{n}} \leq u \big\} }  - \mathbb{E} \left(  \frac{X_{i}}{a_{n}} \mathbf{1}_{ \big\{ \frac{|X_{i}|}{a_{n}} \leq u \big\} }  \right) \bigg\}  \bigg|  > \epsilon
       \bigg)
\\[0.3em]
\nonumber        
&\ + \mathbb{P} \bigg(\max_{ 0 \leq k \leq n } \ \bigg| \sum_{i=1}^{ k } \bigg\{  \frac{X^2_{i}}{a^2_{n}}
       \mathbf{1}_{ \big\{ \frac{|X_{i}|}{a_{n}} \leq u \big\} }  - \mathbb{E} \left(  \frac{X^2_{i}}{a^2_{n}} \mathbf{1}_{ \big\{ \frac{|X_{i}|}{a_{n}} \leq u \big\} }  \right) \bigg\}  \bigg|  > \epsilon \bigg).
 \end{align}
\end{small}
Therefore, from previous conditions it follows that 
\begin{align}
\lim_{u \to 0} \limsup_{n \to \infty} \Pr \biggl( \sup_{t \in [0,1]} \|L_{n}(t) - L_{n}^{(u)}(t)\| > \epsilon \biggr) = 0.   
\end{align}

\newpage

Lastly, it remains to show that the L\'evy process $V$ has a characteristic triple $\left( 0, \nu^{\prime}, \gamma \right)$, which will imply that the limit random variable is equivalent to the case where $\alpha \in (0,1)$ presented in the literature, related to weak convergence of partial sum processes to the space $D[0,1]$ equipped with the $M_1$ topology (see, \cite{basrak2010functional}). To prove the required result we consider the analytical expression of the characteristic function of the $\alpha-$stable random variable $V(1)$. 

The L\'evy process $V$ is the weak limit in the sense of finite dimensional distributions of the partial sum process $V_n$ which is related to the canonical measure in Feller's representation of an infinitely divisible characteristic function (see, \cite{feller1971introduction}) which is also discussed in \cite{davis1995point}. Therefore, in the following derivations we follow the approach presented both by  \cite{basrak2018invariance} and \cite{krizmanic2020joint}, following the derivations given in Theorem 3.2 of \cite{davis1995point}. In other words, we obtain the analytical expression of the partial sum process $V_n$ in terms of its characteristic function.    

Specifically, we denote with $V_n(t)$ the L\'evy process of the first coordinate as below
\begin{align}
V_n(t) := L_{1n}(t)  = 
\begin{cases}
\displaystyle \sum_{i = 1}^{ \floor{nt} } \frac{X_i}{ a_n } & ,\text{if} \ \alpha \in (0,1)
\\
\\
\displaystyle
\sum_{i = 1}^{ \floor{nt} } \frac{X_i}{ a_n } - \floor{nt} \mathbb{E} \left[ \frac{ X_1 }{\alpha_n} \mathbf{1}_{ \left\{ \frac{ X_1 }{\alpha_n} \leq 1 \right\}} \right] & ,\text{if} \ \alpha \in [1,2).
\end{cases}
\end{align}
where
\begin{align}
\label{L1expression}
L_1(\,\cdot\,) = 
\begin{cases}
\displaystyle \sum_{ T_i \leq . } \sum_j P_i \eta_{ij}, & \alpha \in (0,1),
\\
\displaystyle \underset{ u \to 0 }{ \mathsf{lim} } \left( \sum_{ T_i \leq . } \sum_j P_i \eta_{ij} 1_{ \left\{ P_i \left| \eta_{ij} \right| > u  \right\} } - (\,\cdot\,) \int_{  u < \left| x \right| \leq 1 } x \mu (dx) \right), & \alpha \in [1,2),
\end{cases}
\end{align}   
Similarly, we denote with $W_n(t)$ the L\'evy process of the second coordinate as below
\begin{align}
W_n(t) := L_{2n}(t)  = 
\begin{cases}
\displaystyle \sum_{i = 1}^{ \floor{nt} } \frac{X^2_i}{ a_n } & ,\text{if} \ \alpha \in (0,1)
\\
\\
\displaystyle
\sum_{i = 1}^{ \floor{nt} } \frac{X^2_i}{ a_n } - \floor{nt} \mathbb{E} \left[ \frac{ X^2_1 }{\alpha_n} \mathbf{1}_{ \left\{ \frac{ X^2_1 }{\alpha_n} \leq 1 \right\}} \right] & ,\text{if} \ \alpha \in [1,2).
\end{cases}
\end{align}
where
\begin{align}
L_2(\,\cdot\,) = 
\begin{cases}
\displaystyle \sum_{ T_i \leq . } \sum_j P^2_i \eta^2_{ij}, & \alpha \in (0,1),
\\
\displaystyle \underset{ u \to 0 }{ \mathsf{lim} } \left( \sum_{ T_i \leq . } \sum_j P^2_i \eta^2_{ij} 1_{ \left\{ P^2_i \left| \eta_{ij} \right| > u  \right\} } - (\,\cdot\,) \int_{  u < \left| x \right| \leq 1 } x^2 \mu (dx) \right), & \alpha \in [1,2),
\end{cases}
\end{align}

\newpage

\underline{ Case 1: $\alpha \in (0,1)$ }

\begin{itemize}

\item[Step 1.]  For every $\epsilon > 0$, consider the functions $s^{(u)}, w^{(u)}$ and $y^{(u)}$ defined on $\tilde{\ell}_0$ by 
\begin{align}
s^{(u)} ( \tilde{\boldsymbol{x}} ) 
&=
\sum_j x_j  \boldsymbol{1}_{ \left\{ |x_j| > u \right\} }  
\\
w^{(u)} ( \tilde{\boldsymbol{x}} ) 
&=
\underset{ k }{ \mathsf{inf} } \sum_{ j \leq k } x_j \boldsymbol{1}_{ \left\{ |x_j| > u \right\} }  
\\
y^{(u)} ( \tilde{\boldsymbol{x}} ) 
&=
\underset{ k }{ \mathsf{sup} } \sum_{ j \leq k } x_j \boldsymbol{1}_{ \left\{ |x_j| > u \right\} } 
\end{align}
Moreover, define the mapping $ \mathcal{T}_{(u)}: \mathcal{M}_p \big( [0,1] \times \tilde{\ell}_0 \ \left\{  \boldsymbol{0}  \right\} \big) \to D[0,1]$, by setting for $\gamma = \sum_{i=1}^{ \infty } \delta_{  t_i, \tilde{\boldsymbol{x}}^i }$, \begin{align}
\mathcal{T}_{(u)}  \big( \gamma \big) := \left(  \left( \sum_{T_i \leq t} s^{(u)} ( \tilde{\boldsymbol{x}}^i )   \right)_{ T \in [0,1] }, \big\{ T_i: \norm{ \tilde{\boldsymbol{x}}^i }_{ \infty } > u \big\},  \big\{ I(T_i): \norm{ \tilde{\boldsymbol{x}}^i }_{ \infty } > u \big\}  \right)   
\end{align}
where 
\begin{align}
I(T_i) = \sum_{ T_j < T_i } s^{(u)} (\tilde{\boldsymbol{x}}^j) + \left[ \sum_{T_k = T_i}  w^{(u)} (\tilde{\boldsymbol{x}}^k), \sum_{T_k = T_i} y^{(u)} (\tilde{\boldsymbol{x}}^k)     \right].    
\end{align}

Let $\Lambda = \Lambda_{1} \cap \Lambda_{2}$. Moreover, the subsets $\Lambda_1$ and $\Lambda_2$ are defined as below
\begin{multline*}
\label{lampdas}
\Lambda_{1} =
\{ \eta \in \mathbf{M}_{p}([0,1] \times \mathbb{E}) :
\eta ( \{0,1 \} \times \mathbb{E}) = 0 = \eta ([0,1] \times \{ \pm \infty, \pm u \}) \}, \\[1em]
\shoveleft \Lambda_{2} =
\{ \eta \in \mathbf{M}_{p}([0,1] \times \mathbb{E}) :
\eta ( \{ t \} \times (u, \infty]) \cdot \eta ( \{ t \} \times [-\infty,-u)) = 0, \  \text{for all $t \in [0,1]$} \}.
\end{multline*}

\medskip

\item[Step 2.] Define with $W_i = \sum_{ j \in \mathbb{Z} } | Q_{i,j} |$ and $\sum_{i=1}^{ \infty } \delta_{ P_i W_i }$ is a Poisson point process on $(0, \infty]$ with intensity measure given by $\theta \mathbb{E} [ W_i^{\alpha} ] \alpha y^{ - \alpha - 1 } dy$. This definition implies that we can express the point process as below 
\begin{align}
\sum_{i=1}^{ \infty } P_i W_i = \sum_{i=1}^{ \infty }  \sum_{ j \in \mathbb{Z} } Pi |Q_{i,j } | < \infty < \infty, \ \ \text{almost surely},       
\end{align}
Our aim is to prove that the above limit process is an $\alpha-$stable L\'evy process in the topological space $D( [0,1], \mathbb{R} )$ equipped with the $M_1$ topology.

\newpage

Therefore, by defining with $s( \tilde{\boldsymbol{x} } ) = \sum_j x_j$, we obtain that the process can be expressed as 
\begin{align}
V(t) = \sum_{ T_i \leq t  }  s \left( P_i \boldsymbol{Q}_i \right), \ \ \ t \in [0,1]
\end{align}
is almost surely a well-defined element in the topological space $D$ and thus, is an $\alpha-$stable L\'evy process. Consequently we have that the partial sum process is equivalent to 
\begin{align*}
S_{n,u}(t)
&= 
\sum_{i=1}^{ \floor{nt} } \frac{ X_i }{ a_n } \boldsymbol{1}_{ \left\{ \frac{|X_i|}{ a_n } > u \right\} }, \ \ \alpha \in (0,1) \ \text{and} \ t \in [0,1],  
\\
&=  
\sum_{ T_i \leq t } s \left( P_i \boldsymbol{Q}_i \right) 
=
\sum_{ T_i \leq t } \sum_{ j \in \mathbb{Z} } P_i Q_{i,j} \boldsymbol{1}_{ \left\{ | P_i Q_{i,j} | > u \right\} }, \ \ t \in [0,1]. 
\end{align*}
Notice that Step 2 above holds due to an application of the following Corollary.

\medskip

\begin{corollary}
Let $\left\{ r_n \right\}$ be a nonnegative non decreasing integer valued sequence such that $\mathsf{lim}_{ n \to \infty } r_n  \infty$ and $\mathsf{lim}_{ n \to \infty } n / r_n = \infty$, and let $\left\{ a_n \right\}$ be a non decreasing sequence such that $n \mathbb{P} \big( |X_0| > a_n \big) \to 1$. Then, it holds that
\begin{align}
N_n^{ \prime \prime } \sum_{i=1}^{ k_n }  \delta_{ ( i / k_n, ( X_{(i-1) r_n + 1},..., X_{i r_n} ) / a_n ) }  \overset{ d }{ \to } \sum_{i=1}^{ \infty } \delta_{ ( T_i, P_i \boldsymbol{Q}_i )}
\end{align}
where $\delta_{ ( T_i, P_i)} $ is a Poisson point process on $[0,1] \times (0, \infty]$ with intensity measure $Leb \times d ( - \theta y^{-\alpha} )$. 
\end{corollary}

\medskip

\begin{remark}
Notice that the above representation holds if and only if we prove that all the previous steps hold. In other words, we need to prove that in our settings a similar limit result holds as in Theorem 4.5 in the paper of \cite{basrak2018invariance}
\end{remark}

\medskip

\item[Step 3.] The next step, is to show that the partial sum process under consideration is a tight process in the topological space $D ( [0,1], \mathbb{R} )$. Notice that since two marginal sequences of processes are tight, the joint sequences of processes must also be tight.

\newpage

\item[Step 4.] The last step is to show that the original partial sum process $S_n$; and therefore $V_n$, since $\alpha \in (0,1)$ also converges in distribution to $V^{\prime}$, which implies that we can use the metric distance to prove that it converges in probability to zero as $n \to \infty$.

\end{itemize}

Therefore, since the limit variable $V$ that corresponds to the limit of the partial sum process $V_n$ is an $\alpha-$stable random variable, then it has the following characteristic function representation
\begin{align}
\mathbf{E} \left[ e^{iz V(1)} \right] = 
\begin{cases}
\displaystyle \mathsf{exp} \left\{ i \tau z - c|z|^{\alpha} \left( 1 - i \beta \mathsf{sign}(z) \mathsf{tan} \left( \frac{\pi \alpha}{2} \right) \right)   \right\}, & \alpha \neq 1
\\
\\
\displaystyle \mathsf{exp} \left\{ i \tau z - c|z| \left( 1 + i \beta \frac{2}{\pi} \mathsf{sign}(z) \mathsf{log} |z| \right) \right\}, & \alpha = 1
\end{cases}
\end{align}
where $c >0$, $\beta \in [-1,1]$ and $\tau \in \mathbb{R}$, with $V(1)$ as expression \eqref{L1expression} in Theorem 3.4. By Remark 3.2 in \cite{davis1995point}, the \textcolor{blue}{scale parameter $c$} and the \textcolor{blue}{symmetry parameter $\beta$} are defined with the following coefficients respectively, (see also, \cite{feller1971introduction} pages 568-570). 
\begin{align}
\textcolor{blue}{c} &= 
\begin{cases}
\big( c_{+} + c_{-} \big) \frac{ \displaystyle \Gamma ( 3 - \alpha) }{ \displaystyle \alpha (\alpha - 1)} \mathsf{cos} \left( \frac{\pi \alpha}{2} \right), & \text{if} \ \alpha \neq 1, 
\\
\big( c_{+} + c_{-} \big) \displaystyle \frac{\pi}{2}, & \text{if} \ \alpha = 1
\end{cases}
\\
\textcolor{blue}{\beta} &= \frac{  c_{+} - c_{-}  }{ c_{+} + c_{-} },
\end{align}
where
\begin{align}
c_{+} 
&= 
\frac{ \alpha }{ 2 - \alpha } \ \underset{ u \to 0 }{ \mathsf{lim} } \int_0^{ \infty } \mathbb{P} \left(  \sum_{j=1}^{\infty} y Q_{1j} \mathbf{1}_{ \left\{ u, \infty \right\} } ( y | Q_{1j} | )  > 1  \right) \gamma \alpha y^{- \alpha - 1 } dy  
\\
c_{+} 
&= 
\frac{ \alpha }{ 2 - \alpha } \ \underset{ u \to 0 }{ \mathsf{lim} } \int_0^{ \infty } \mathbb{P} \left(  \sum_{j=1}^{\infty} y Q_{1j} \mathbf{1}_{ \left\{ u, \infty \right\} } ( y | Q_{1j} | )  < - 1  \right) \gamma \alpha y^{- \alpha - 1 } dy   
\end{align}

Moreover, the \textcolor{blue}{location parameter $\tau$} is defined as below
\begin{align}
\tau 
=
\begin{cases}
\underset{ u \to 0 }{ \mathsf{lim} } \ \displaystyle \int_{ - \infty }^{ \infty } \frac{ \displaystyle \mathsf{sin} x }{ x^2 } \textcolor{blue}{M_{u}}  (dx) - b(u,1], & \ \text{if} \ \alpha = 1,
\\
\underset{ u \to 0 }{ \mathsf{lim} } \ \displaystyle \int_{ - \infty }^{ \infty } x^{-1} \textcolor{blue}{M_{u}} (dx) - b(u,1], & \ \text{if} \ \alpha > 1,
\end{cases}
\end{align}
where \textcolor{blue}{$M_u(dx)$} is the canonical measure in Feller's representation of an infinitely divisible characteristic function given by 
\begin{align}
M_u(dx) = x^2 \int_0^{\infty} \mathbb{P} \left(  \sum_{j=1}^{\infty} y Q_{1j} \mathbf{1}_{ \left\{ u, \infty \right\} } ( y | Q_{1j} | )  \in dx  \right) \gamma \alpha y^{- \alpha - 1 } dy    
\end{align}

\newpage

\underline{Case 2: $\alpha \in [1,2)$}

\

Furthermore, in the case when $\alpha \in [1,2)$ we need to introduce centering, and therefore we define the c\'adl\'ag process $V_n$ by setting $t \in [0,1]$ for
\begin{align}
S_{n,u}(t) = \sum_{i=1}^{ \floor{nt} } \frac{ X_i }{ a_n } \mathbf{1}_{ \left\{ \frac{|X_i|}{ a_n } > u \right\} }   
\end{align}
and the corresponding partial sum process with a centering sequence given by 
\begin{align}
V_{n,u}(t) = S_{n,u}(t)  - \floor{ nt } \mathbf{E} \left[ \frac{X_1}{ a_n } 1_{ \left\{ u <  \frac{|X_i|}{ a_n } \leq 1 \right\} } \right].     
\end{align}

\medskip

\begin{lemma}
Let $\alpha \in [1,2)$ and let the convergence of Theorem 4.5 hold. Then there exists an $\alpha-$stable L\'evy process $V$ on $[0,1]$ such that, as $\epsilon \to 0$, the process $V_{\epsilon}$ defined below
\begin{align}
V_{\epsilon}(t) = \sum_{T_i \leq t} s^{\epsilon} ( P_i \boldsymbol{Q}_i ) - t \int_{ \left\{ x : \epsilon < |x| \leq 1 \right\} }  x \mu (dx).   \end{align}
converges uniformly almost surely (along some subsequence) to $V$.
\end{lemma}

\begin{proof}
More precisely, we show that for all $\epsilon > 0$
\begin{align}
\int_{ \left\{ x : \epsilon < |x| \leq 1 \right\} } x \mu (dx) = \theta \int_0^{\infty} \mathbb{E} \left[ y \sum_{ y \in \mathbb{Z} } Q_j \boldsymbol{1}_{ \left\{ \epsilon < y |Q_j| \leq 1 \right\} } \right] \alpha y^{ - \alpha - 1 } dy.   
\end{align}
Notice that it holds that 
\begin{align}
\theta \mathbb{E} \left[  \sum_{ j \in \mathbb{Z} } Q_j \left| Q_j \right|^{\alpha - 1} \right] = 2p - 1. 
\end{align}
Moreover, by Fubini's theorem, if $\alpha > 1$, we have that 
\begin{align*}
\theta \int_0^{ \infty } \mathbb{E} \left[ y \sum_{j \in \mathbb{Z} } Q_j \boldsymbol{1}_{ \left\{ \epsilon < y |Q_j| \leq 1 \right\} } \right] \alpha y^{- \alpha - 1 } dy
&= \alpha \theta  \mathbb{E} \left[ \sum_{j \in \mathbb{Z} }  Q_j \int_{\epsilon |Q_j|^{-1} }^{ |Q_j|^{-1} }  y^{-\alpha} dy \right]
\\
&= \frac{ \alpha }{  \alpha - 1 } \left( \epsilon^{ - \alpha + 1 } - 1 \right) \theta \mathbb{E} \left[ \sum_{j \in \mathbb{Z} }  Q_j | Q_j |^{ \alpha - 1} \right]
\\
&=
\frac{ \alpha }{  \alpha - 1 } \left( \epsilon^{ - \alpha + 1 } - 1 \right)  (2p - 1). 
\end{align*}

\newpage

Hence, for all $t \in [0,1]$ we have that 
\begin{align}
V_{\epsilon}(t) = \sum_{T_i \leq t} s^{\epsilon} ( P_i \boldsymbol{Q}_i ) - t \theta \int_0^{\infty}  \mathbb{E} \left[ y \sum_{j \in \mathbb{Z} } Q_j \boldsymbol{1}_{ \left\{ \epsilon < y |Q_j| \leq 1 \right\} } \right] \alpha y^{- \alpha - 1 } dy 
\end{align}
Furthermore, notice that we can define with $W = \sum_{j \in \mathbb{Z} } |Q_j |$ and $W_i = \sum_{j \in \mathbb{Z} } | Q_{i,j} |$ such that $\left\{ W_i, i \geq 1 \right\}$ is a sequence of $\textit{i.i.d}$ random variables with the same distribution as $W$.

Therefore, for every $\delta > 0$, there are almost surely at most finitely many points $P_i W_i$ such that it holds that $P_i W_i > \delta$. For every $\delta, \epsilon > 0$, define with 
\begin{align}
m_{\epsilon, \delta} = \theta \int_0^{\infty} \mathbb{E} \left[ y \sum_{j \in \mathbb{Z} } Q_j \boldsymbol{1}_{ \left\{ \epsilon < y |Q_j| \leq 1, \delta < y W \right\} } \right]  \alpha y^{- \alpha - 1 } dy    
\end{align}

\begin{remark}
Notice that by the Theorem 3.1 of \cite{basrak2018invariance}  it holds that the finite dimensional distributions of $V_{0, \delta}$ converge to those of an $\alpha-$stable L\'evy process. Since the topological space $D( [0,1])$ is complete under its metric, we can obtain the existence of the stochastic process $V = \left\{ V(t), t \in [0,1] \right\}$ with continuous paths (trajectories with independent increments), in $D([0,1] )$ almost surely and such that $\mathsf{lim}_{ k \to \infty } \norm{ V_{0, \delta_k }  - V }_{\infty} = 0$  almost surely.  
\end{remark}

Therefore, it only remains to prove that for all $u > 0$, we have that 
\begin{align}
\underset{ \delta \to 0 }{ \mathsf{lim} } \  \underset{ \epsilon \to 0 }{ \mathsf{lim sup} } \ \mathbb{P} \big(  \norm{ V_{\epsilon} - V_{\epsilon, \delta} }_{\infty} > u \big) = 0.    
\end{align}
In particular, proving that the above result holds, then it implies that $\norm{ V_{\epsilon} - V_{\epsilon, \delta} }_{\infty} \to 0$, in probability and hence that, along some subsequence, $V_{\epsilon}$ converges to $V$ unifornmly almost surely.

Moreover, since for $\delta \leq 1$, $y W = \sum_{ j \in \mathbb{Z} } y |Q_j| \leq \delta$ implies that $y |Q_j | \leq \delta \leq 1$ for all $j \in \mathbb{Z}$,  
\begin{align}
V_{\epsilon} (t) - V_{\epsilon, \delta} (t) = \sum_{ T_i \leq t } \sum_{ j \in \mathbb{Z} } Pi Q_{i,j} \boldsymbol{1}_{ \left\{ \epsilon < P_i |Q_j| , P_i W_i \leq \delta \right\} }  - t \theta \int_0^{\infty} \mathbb{E} \left[ y \sum_{j \in \mathbb{Z} Q_j } \boldsymbol{1}_{ \left\{ \epsilon < P_i |Q_j| , P_i W_i \leq \delta \right\} } \right]  
\end{align}

\end{proof}

Therefore, in order to prove that indeed the coefficients that correspond to the characteristic function of the partial sum process converge to the L\'evy process $V$ with characteristic function and corresponding coefficients that are part of the triple, we need to employ the derivations used in Lemma 6.3 of \cite{davis1995point}.

\newpage

Thus, the components of the limiting process $L = \left( L_1, L_2 \right)$ will be expressed as functionals of the limiting process $N = \sum_i \sum_j \delta_{(T_i, P_i \eta_{ij} )}$, and the description of the characteristic triple of $L_1$ and the exponent measure of $L_2$ will be in terms of the measures $\nu_1$ and $\nu_2$ on $\mathbb{R}$ defined by 
\begin{align*}
\nu_1 (dx) 
&= 
\big( c_{+} \mathbf{1}_{ ( 0, + \infty )}(x) + c_{-} \mathbf{1}_{ (- \infty,0)}(x) \big) \theta \alpha |x|^{- \alpha - 1} dx  
\\
\nu_2( dx) 
&= 
r \theta \alpha x^{- \alpha - 1} \mathbf{1}_{ ( 0, + \infty )}(x)  dx
\end{align*}
We use the following notation 
\begin{align}
\textcolor{blue}{x^{ \alpha^{\star} } := x |x|^{\alpha - 1}}  \equiv |x|^{\alpha} \left( \mathbf{1}_{(0,\infty}(x) - \mathbf{1}_{(- \infty,0)}(x) \right)
\end{align}

\medskip

The following parameters were computed by Theorem 3.2 in \cite{davis1995point}. Thus, we define with 
\begin{align}
c := \theta \mathbb{E} \left[ \left| \sum_j \eta_j \right|^{\alpha} \right], \ \ \ \beta 
:= 
\frac{ \mathbb{E} \left[ \left( \sum_j \eta_j \right)^{ \alpha^{\star} } \right] }{ \mathbb{E} \left[ \left| \sum_j \eta_j \right|^{\alpha} \right] }
=      
\frac{ \mathbb{E} \left[ \left( \sum_j \eta_j \right) \left| \sum_j \eta_j \right|^{\alpha - 1} \right]  }{ \mathbb{E} \left[ \left| \sum_j \eta_j \right|^{\alpha} \right] }
\end{align}

\medskip

\begin{align}
\tau = 
\begin{cases}
\displaystyle  ( \alpha - 1 )^{-1} \alpha \theta \mathbb{E} \left[ \left( \sum_j \eta_j \right)^{ \alpha^{\star} } \right], & \alpha > 1, 
\\
\displaystyle  \theta \left( s \mathbb{E} \left[ \sum_j \eta_j \right] - \mathbb{E} \left[ \sum_j \eta_j \mathsf{log}\left( \left| \sum_i \eta_i \eta_j^{-1} \right| \right) \right]      \right), & \alpha = 1,
\end{cases}
\end{align}
with $s = \int_0^{\infty} \left( \mathsf{sin}(x) - x \mathbf{1}_{(0,1]}(x)  \right) x^{-2} dx$. Then, the characteristic triple of the process $V$ is of the form $( 0, \sigma, b )$, where $\sigma(dx) = \left( c_1 \mathbf{1}_{(0,1]}(x)  + c_2 \mathbf{1}_{(- \infty,0)}(x) \right) |x|^{-(\alpha+1)} dx$ and $b = \tau - d$, with 

\begin{align}
c_1 = 
\begin{cases}
\frac{ \displaystyle - c ( 1 + \beta ) }{ \displaystyle 2 \Gamma(- \alpha) \mathsf{cos} ( \pi \alpha / 2 ) }, & \alpha > 1
\\
\\
\frac{ \displaystyle c(1 + \beta )}{ \displaystyle \pi}, & \alpha = 1
\end{cases}
\ \ \ \
\text{and} 
\ \ \ \
c_2 = 
\begin{cases}
\frac{ \displaystyle - c ( 1 - \beta ) }{ \displaystyle 2 \Gamma(- \alpha) \mathsf{cos} ( \pi \alpha / 2 ) }, & \alpha > 1
\\
\\
\frac{ \displaystyle c(1 - \beta )}{ \displaystyle \pi}, & \alpha = 1
\end{cases}
\end{align} 

\medskip

Thus, for a complete proof we shall also show that $b = \gamma$ which will then imply that the L\'evy process $V$ has characteristic triple $( 0, \nu_1, \gamma_1 )$.  We leave this step for the next version of this draft. 

\end{proof}

\newpage

\subsection{Proof of Theorem \ref{t:functSN1}}

\begin{proof}
Repeating the arguments from the proof of Theorem~\ref{t:functconvergence}, but with the summation functional $\widetilde{\Phi}^{(u)} \colon \mathbf{M}_{p}([0,1] \times \mathbb{E}) \to D([0,1], \mathbb{R}^{2})$,
\begin{align}
\widetilde{\Phi}^{(u)} \Big( \sum_{i}\delta_{(t_{i}, x_{i})} \Big) (t)
=  
\Big( \sum_{t_{i} \leq t}x_{i}\, \boldsymbol{1}_{\{u < |x_{i}| < \infty \}},  \sum_{i} x_{i}^{2}\, \boldsymbol{1}_{\{u < |x_{i}| < \infty \}}  \Big), \qquad t \in [0,1],    
\end{align}
instead of $\widetilde{\Phi}^{(u)}$, we obtain, as $n \to \infty$,
\begin{equation}
\label{e:mainconvSN}
\widetilde{L}_{n}(t) 
:= \bigg( \sum_{i = 1}^{\lfloor n t \rfloor} \frac{X_{i}}{a_{n}},  \sum_{i = 1}^{n} \frac{X_{i}^{2}}{a_{n}^{2}} \bigg)
    \dto \widetilde{L} :=  \bigg( \sum_{T_{i} \le t} \sum_{j}P_{i}\eta_{ij}, \sum_{i} \sum_{j}P_{i}^{2}\eta_{ij}^{2} \bigg)
\end{equation}
in $D([0,1], \mathbb{R}^{2})$ with the weak $M_{1}$ topology. Similar to Lemma~\ref{l:M1div} one can show that the function $g \colon D([0,1], \mathbb{R}) \times C^{\uparrow}_{0}([0,1], \mathbb{R}) \to D([0,1], \mathbb{R})$ defined by
\begin{align*}
g(x,y) = \frac{x}{\sqrt{y}}, y > 0 \ \ \ \ \text{and} \ \ \ \ g(x,y) = 0, y \leq 0,
\end{align*}
is continuous
when $D([0,1], \mathbb{R}) \times C^{\uparrow}_{0}([0,1], \mathbb{R})$ is endowed with the weak $M_{1}$ topology and $D([0,1], \mathbb{R})$ is endowed with the standard $M_{1}$ topology. Since
\begin{align}
\mathbb{P} \left(  \widetilde{L} \in D([0,1], \mathbb{R}) \times C^{\uparrow}_{0}([0,1], \mathbb{R}) \right)=1,    
\end{align}
an application of the continuous mapping theorem yields $g(\widetilde{L}_{n}) \dto g(\widetilde{L})$ as $n \to \infty$, that is
\begin{align}
g \left(  \widetilde{L}_{1n}(t),  \widetilde{L}_{2n}(1)   \right) 
:= 
\frac{ \widetilde{L}_{1n}(t)  }{ \sqrt{\widetilde{L}_{2n}(1)} }
\equiv
\frac{ \displaystyle \sum_{i = 1}^{\lfloor n t \rfloor} \frac{X_i }{ a_n}  }{ \displaystyle \sqrt{ \sum_{i = 1}^n  \frac{X^2_i }{ a_n} } }
\dto \frac{ \displaystyle \sum_{T_{i} \le t} \sum_{j}P_{i}\eta_{ij}}{ \displaystyle \sqrt{\sum_{i} \sum_{j}P_{i}^{2}\eta_{ij}^{2}}} = \frac{L_{1}(t)}{\sqrt{L_{2}(1)}} , \ \ \alpha \in (0,1).  
\end{align}
in $D([0,1], \mathbb{R})$ endowed with the $M_{1}$ topology. 
\end{proof}

\subsection{Proof of Theorem \ref{t:functSN2}}

Notice that in the case of $\alpha \in [1,2)$ we still consider the joint functional convergence of $L_n(t) := \big( L_{1n}(t), L_{2n}(t) \big)$. However, a key aspect in this proof is that we employ the convergence of these coordinates to obtain a weak limit result in $D[0,1]$ endowed with the $M_1$ topology for the corresponding self-normalized partial sum process. Therefore, following similar arguments to the proof of Theorem~\ref{t:functconvergence}, but in this case using the summation functional $ \widetilde{\Phi}^{(u)} \colon \mathbf{M}_{p}([0,1] \times \mathbb{E} ) \to D([0,1], \mathbb{R}^{2})$, we obtain that
\begin{align*}
\widetilde{\Phi}^{(u)} \Big( \sum_{i}\delta_{(t_{i}, x_{i})} \Big) (t)
=  
\Big( \sum_{t_{i} \leq t}x_{i}\, \boldsymbol{1}_{\{u < |x_{i}| < \infty \}},  \sum_{i} x_{i}^{2}\, \boldsymbol{1}_{\{u < |x_{i}| < \infty \}}  \Big), \qquad t \in [0,1],
\end{align*}
instead of $\widetilde{\Phi}^{(u)}$, we obtain, as $n \to \infty$,
\begin{align}
\label{e:mainconvSN}
\widetilde{L}_{n}(t)
&:= 
\bigg( \sum_{i = 1}^{\lfloor nt \rfloor} \frac{X_{i}}{a_{n}} - \floor{nt} b_{1n}^{(u)},  \sum_{i = 1}^{n} \frac{X_{i}^{2}}{a_{n}^{2}} - n b_{2n}^{(u)} \bigg)
\nonumber
\\
&\dto 
\widetilde{L} :=  \bigg( \sum_{T_{i} \le t  } \sum_{j} P_{i}\eta_{ij} \mathbf{1}_{ \left\{ P_i | \eta_{ik} | > u \right\} } - t b_1, \sum_{i} \sum_{j}P_{i}^{2}\eta_{ij}^{2} \mathbf{1}_{ \left\{ P^2_i | \eta_{ij}^2 | > u \right\} } - b_2 \bigg)
\end{align}
in $D([0,1], \mathbb{R}^{2})$ with the weak $M_{1}$ topology. Similar to Lemma~\ref{l:M1div} one can show that the function $g \colon D([0,1], \mathbb{R}) \times C^{\uparrow}_{0}([0,1], \mathbb{R}) \to D([0,1], \mathbb{R})$ defined by
\begin{align*}
g(x,y) = \frac{x}{\sqrt{y}}, y > 0 \ \ \ \ \text{and} \ \ \ \ g(x,y) = 0, y \leq 0,
\end{align*}
is continuous when $D([0,1], \mathbb{R}) \times C^{\uparrow}_{0}([0,1], \mathbb{R})$ is endowed with the weak $M_{1}$ topology and $D([0,1], \mathbb{R})$ is endowed with the standard $M_{1}$ topology. Since
\begin{align}
\mathbb{P} \left(\widetilde{L} \in D([0,1], \mathbb{R}) \times C^{\uparrow}_{0}([0,1], \mathbb{R}) \right) =1,
\end{align}
an application of the continuous mapping theorem yields $g(\widetilde{L}_{n}) \dto g(\widetilde{L})$ as $n \to \infty$, that is
\begin{align*}
g \left(  \widetilde{L}_{1n}(t),  \widetilde{L}_{2n}(1)   \right) 
:= 
\frac{ \widetilde{L}_{1n}(t)  }{ \sqrt{\widetilde{L}_{2n}(1)} }
&\equiv 
\frac{ \displaystyle  \sum_{i = 1}^{\lfloor nt \rfloor} \frac{X_{i}}{a_{n}} - \floor{nt} b_{1n}^{(u)} }{ \displaystyle   \sqrt{ \sum_{i = 1}^{n} \frac{X_{i}^{2}}{a_{n}^{2}} - n b_{2n}^{(u)} } }
= 
\frac{ \displaystyle  \sum_{i = 1}^{\lfloor nt \rfloor} \frac{X_{i}}{a_{n}} - \floor{nt} \mathbb{E} \left[ \frac{X_1}{ a_n } \mathbf{1}_{ \left\{ \textcolor{red}{u < } \frac{ | X_1 | }{ a_n } \leq 1 \right\} } \right] }{ \displaystyle  \sqrt{  \sum_{i = 1}^{n} \frac{X_{i}^{2}}{a_{n}^{2}} - n \mathbb{E} \left[ \frac{X^2_1}{ a^2_n } \mathbf{1}_{ \left\{ u < \frac{ X_1^2 }{ a_n } \leq 1  \right\} } \right]  } }, \ \  \alpha \in [1,2)
\\
&\dto 
\frac{ \displaystyle \underset{ u \to 0 }{ \mathsf{lim} } \left\{ \sum_{T_{i} \le t} \sum_{j}P_{i}\eta_{ij} \mathbf{1}_{ \left\{ P_i | \eta_{ik} | > u \right\} } - t \int_{ u < |x| \leq 1} x \mu (dx)  \right\} }{ \displaystyle \sqrt{\sum_{T_i \leq 1 } \sum_{j}P_{i}^{2}\eta_{ij}^{2} \mathbf{1}_{ \left\{ P^2_i | \eta_{ij}^2 | > u \right\} } - \int_{ u < x^2 \leq 1} x^2 \mu (dx)  } } 
= \frac{ \widetilde{L}_{1} (\,t\,)}{\sqrt{ \widetilde{L}_{2}(1)}}, \ \text{as} \ n \to \infty.   
\end{align*}
in $D([0,1], \mathbb{R})$ endowed with the $M_{1}$ topology. 

\medskip

The last equality follows by the joint weak convergence of the functionals $\widetilde{L}_{1n} (\,\cdot\,) \dto \widetilde{L}_1 (\,\cdot\,)$ and $\widetilde{L}_{2n} (\,\cdot\,) \dto \widetilde{L}_2 (\,\cdot\,)$ and the fact that $g(\widetilde{L}_{n}) \dto g(\widetilde{L})$ as $n \to \infty$. In particular, note that the limit both on the numerator as well as in the denominator that correspond to the case in which $\alpha \in [1,2)$ hold almost surely uniformly in $[0,1]$. Moreover, in the case of $\alpha = 1$, additionally we assume that Conditions \eqref{cond2} and \eqref{cond3}. These two conditions although challenging to prove in the case of dependent data, ensure that our main results hold for more general dependence structures as in the case of point processes. A special type of dependence one can consider is the case of $\rho-$mixing which those two conditions can be easily established.

\end{appendix}

\newpage 
   
\bibliographystyle{apalike}
\bibliography{myreferences1}

\end{document}